\documentclass[times,doublespace]{amsart}
\usepackage{txfonts}
\usepackage{graphicx}
\usepackage{amssymb}
\usepackage{epstopdf}
\usepackage[a4paper, total={6in, 9in}]{geometry}

\usepackage{latexsym,amsmath,amsfonts,amsthm,mathrsfs,amssymb,cite}
\usepackage[usenames]{color}
\usepackage{graphicx,epsfig}
\usepackage{amscd}
\usepackage{amsbsy}
\usepackage{bm}
\usepackage{color}
\usepackage{cite}
\usepackage[colorlinks=true,urlcolor=blue,
citecolor=red,linkcolor=blue,linktocpage,pdfpagelabels]{hyperref}
\newtheorem{thm}{Theorem}[section]

\newtheorem{lem}{Lemma}[section]

\theoremstyle{definition}
\newtheorem{defn}{Definition}[section]
\newtheorem{asm}{Assumption}[section]

\newtheorem{rem}{Remark}[section]
\numberwithin{equation}{section}

\numberwithin{equation}{section}

\newcounter{saveeqn}

\def\nm{\noalign{\medskip}}

\renewcommand{\(}{\left(}
\renewcommand{\)}{\right)}

\newcommand{\Gl}{\lambda}

\newcommand{\Gs}{\sigma}

\newcommand{\vp}{\phi}

\newcommand{\Kcal}{\mathcal{K}}

\newcommand{\Scal}{\mathcal{S}}
\newcommand{\Dcal}{\mathcal{D}}

\newcommand{\Ocal}{\mathcal{O}}

\newcommand{\ds}{\displaystyle}

\newcommand{\pd}[2]{\frac {\p #1}{\p #2}}

\newcommand{\RR}{\mathbb{R}}

\newcommand{\p}{\partial}

\DeclareMathAlphabet{\itbf}{OML}{cmm}{b}{it}

\title[Inverse conductivity problem in layered structures]{Inverse conductivity problem with one measurement: Uniqueness of multi-layer structures}

\author{Lingzheng Kong}
\address{School of Mathematics and Statistics, Central South University, Changsha 410083, China}

\email{math\_klz@csu.edu.cn, math\_klz@163.com}
\thanks{
$^*$ Corresponding author: youjundeng@csu.edu.cn, dengyijun\_001@163.com
}

\author{Youjun Deng$^*$}
\address{School of Mathematics and Statistics, Central South University, Changsha 410083, China}
\email{youjundeng@csu.edu.cn, dengyijun\_001@163.com}

\author{Liyan Zhu}
\address{School of Mathematics and Statistics, Central South University, Changsha 410083, China}
\email{math\_zly@csu.edu.cn}

\date{} 
\begin{document}
\maketitle

\begin{abstract}
In this paper, we study the recovery of multi-layer structures in inverse conductivity problem by using one measurement. First, we define the concept of Generalized Polarization Tensors (GPTs) for multi-layered medium and show some important properties of the proposed GPTs. With the help of GPTs, we present the perturbation formula for general multi-layered medium. Then we derive the perturbed electric potential for multi-layer concentric disks structure in terms of the so-called \emph{generalized polarization matrix}, whose dimension is the same as the number of the layers. By delicate analysis, we derive an algebraic identity involving the geometric and material configurations of multi-layer concentric disks. This enables us to reconstruct the multi-layer structures by using only one \emph{partial-order} measurement.


\end{abstract}

\maketitle

\section{Introduction}
Consider the conductivity problem
\begin{equation}\label{eq:mainmd02}
\left\{
\begin{array}{ll}
\nabla\cdot ((\sigma\chi (A) +\chi(A_0))\nabla u) =0, & \mbox{in} \quad \RR^d,\\
u-H=\Ocal(|x|^{1-d}), & \mbox{as}\quad |x|\rightarrow \infty,
\end{array}
\right.
\end{equation}
where $d=2, 3$ and $A$ is the inclusion embedded in $\RR^d$  with a $C^{1,\eta}\, (0<\eta<1)$ smooth boundary $\p A$, $A_0=\RR^d\setminus\overline{A}$ is the background space,
$\chi$ denotes the characteristic function.
The medium parameter is characterised by the
conductivity which is normalised to be 1 in $A_0$ and is assumed to be $\sigma\in \RR^+$ and $\sigma \neq 1$ in $A$.
The background  electrical potential ${H}$ is a harmonic function in $\RR^d$, and $u$ represents the total electric potential.
In practical applications, the conductivity $\sigma$ might not be homogeneous and usually the inclusion can be modeled as a multi-layer structure. The multi-layer structure, that is a nested body consisting of piecewise homogeneous layers, occurs in many cutting-edge applications such as medical imaging, remote sensing, geophysics, pavement design and invisibility cloaking \cite{AKLL11,AKLL13,AKLLY13,BLip2020,DKLZArxiv,DLLarma2019,DLLarma2020,YMY}.

The inverse conductivity problem can be defined as finding the inclusion $A$ and its conductivity $\sigma$ from given $H$ and boundary measurement.
By using infinitely many measurements or from the Newmann-to-Dirichlet map, the unique recovery results were obtained in \cite{ADKL14,APLcpde2005,BUcpde1997,KVcpam1985,AINam1996,SUam1987}.
While if only finitely many measurements are available, the unique recovery is related to the shape of the inclusion, and the global uniqueness was  obtained only for convex polyhedrons and balls in
$\RR^3$ and for polygons and disks in $\RR^2$, we  refer to \cite{BFSpams1994,FIiumj1989,IPip1990,KSsiam1990,JKSjfaa1996,DLbook2023}. We also refer to \cite{DLLarma2019,DLLarma2020,DLUjde2019,LZip2006,ARpams2005,HK04book,BLip2020,CSijam1983,LLSSsiam2013} for uniqueness results in optics and acoustics.
In this paper, we consider the uniqueness recovery for the inclusion of  multi-layer types, and we only need to use one measurement to locate the inclusion and reconstruct its conductivity distribution.
Such multi-layer structures have been proposed for achieving the so-called GPTs vanishing structures and hence cloaking devices with enhanced invisibility effects via the transformation approach; see \cite{AAHWY17,LTWW21,AKLL11,AKLLY13,AKLL13}, and for achieving surface localized resonance structures by allowing the presence of negative materials, see \cite{DLbook2024,DKLZArxiv,DFArxiv,FDL15}.

In previous works on inverse conductivity problem with one measurement, the main focus is on how to recover the shape of the inclusion by a given constant conductivity $\sigma$. This can be regarded as a one-layer structure.
So far, only a few special types of inclusion, such as disk and ball, polyhedral and polygon, have been proved to be reconstructed by using one measurement.
In the present paper, instead of considering the recovery of the shape, we consider the recovery of the conductivity distribution. Particularly in \cite{FDipi2018}, the authors studied the recovery of conductivity with the
number of layers being 1 or 2.
Motivated by the above works, we consider the recovery of the conductivity distribution within much more general layered structures.
The number of layers can be arbitrary and the material parameters in each layer may be different, though uniform.
The multi-layer structure can be regarded as a special case of general inhomogeneous inclusions. In practical applications,  wave measurement devices are usually deployed far away from the target. Based on this, we shall make use the asymptotic analysis, transmission condition and unique continuation theorem to first locate the multi-layer structure of general shape by using one measurement. We then consider the uniqueness recovery of structure together with the conductivity for multi-layer concentric disks by using one \emph{partial-order} measurement (see Definition \ref{th:expan2}) on some given surfaces.
We derive the perturbed electric potential outside the multi-layer concentric disks in terms of the so-called \emph{generalized polarization matrix} (see (\ref{eq:matP01})), whose dimension is the same as the number of the layers. By delicate analysis, we derive an algebraic identity involving the conductivity. Then by inverting those algebraic  identities using algebraic analysis techniques, we obtain the desired unique recovery results.

The rest of the paper is organized as follows. In Section \ref{sec22}, we introduce the layer potential technique. In Section \ref{sec:03}, we are devoted to defining the Generalized Polarization tensors for multi-layered medium and show some important properties of such GPTs.
In Section \ref{sec2},
we first establish the integral representation of the solution to the conductivity transmission problem within multi-layer structures by using the layer potential techniques. Then we derive the asymptotic expansion of the perturbed electric potential and locate the multi-layer structure by using the first-order polarization tensor.
Section \ref{sec4} is devoted to  reconstructing the conductivity value for multi-layer concentric disks by virtue of \emph{generalized polarization matrix}. Section \ref{s5} contains some conclusion remarks.

\section{Layer potential technique}\label{sec22}
In this section, we shall introduce the layer potentials for Laplacian and prove a decomposition formula of the solution to the conductivity transmission problem (\ref{eq:mainmd02}).
Let $\Gamma_1:=\p A$ and
let the interior of $A$ be divided by means of closed and nonintersecting  $C^{1,\eta}$ surfaces $\Gamma_k$ $(k = 2, 3,...,N)$  into subsets (layers)  $A_k$ $(k = 1, 2,...,N)$.  Each $\Gamma_{k-1}$ surrounds $\Gamma_k$ $(k=2,3,\ldots,N)$.  The regions $A_k$ $(k=1,2,\ldots,N)$ stand for homogeneous media. Assume that
\begin{equation}\label{eq:paracho01}
\sigma(x)=\sigma_k, \quad x\in A_k, \quad k=1, 2, \ldots, N.
\end{equation}
It is nature that the solution $u$ to the conductivity problem (\ref{eq:mainmd02}), with the multi-layer structure defined above, satisfies the transmission conditions
\begin{equation}\label{eq:transmscds}
u|_{+}=u|_{-}\quad\mbox{and}\quad\sigma_{k-1}\frac{\p u}{\p \nu_k}|_{+}=\sigma_{k}\frac{\p u}{\p \nu_k}|_{-}\quad\mbox{on}\quad \Gamma_k,\quad k=1,2,\ldots,N,
\end{equation}
where we used the notation $\nu_k$  to indicate the outward normal on $\Gamma_k$ and $$
\left.w\right|_{\pm}(x)=\lim _{h \rightarrow 0^{+}} w(x \pm h {\nu}),  \quad x \in \Gamma_k,
$$
for an arbitrary function $w$.

Let  $\Gamma$ be a $C^{1,\eta}$ surface. Let $H^{s}(\Gamma)$, for $s\in\RR$,  be the usual $L^2$-Sobolev space and let
$$H^{s}_0 (\Gamma):=\left\{\phi \in H^{s}(\Gamma) : \int_{\Gamma} \phi =0\right\}.$$
For $s=0$, we
use the notation $L^2_0(\Gamma)$.
Let $ G$ be the fundamental solution to the Laplacian in $\RR^d$, that is given by
\[
G(x) =\left\{
\begin{array}{ll}
\frac{1}{2 \pi} \ln |x|, & d=2,\\
\frac{1}{(2-d) \omega_{d}}|x|^{2-d}, & d \geqslant 3,
\end{array}
\right.
\]
where $\omega_{d}$ is the area of the unit sphere in $\RR^d$.
We denote by $\Scal_{\Gamma} : H^{-1/2}(\Gamma)\rightarrow H^{1}(\RR^d)$ the single layer potential operator
\[
\Scal_{\Gamma}[\varphi](x):=\int_{\Gamma}G(x-y)\varphi(y)~\mathrm{d} s(y), \quad x\in \RR^d,
\]
and the double layer potential
$\mathcal{D}_{\Gamma}:H^{1/2}(\Gamma)\rightarrow H^{1}(\RR^d\setminus\Gamma)$ given by
\[
\mathcal{D}_{\Gamma}[\varphi](x):=\int_{\Gamma}\frac{\p}{\p \nu_y}G(x-y)\varphi(y)~\mathrm{d} s(y), \quad x\in \RR^d\setminus\Gamma,
\]
and $\Kcal_{\Gamma} : H^{1/2}({\Gamma})\rightarrow H^{1/2}({\Gamma})$ the Neumann-Poincar\'e (NP) operator
\begin{equation}\label{NPoperator}
\Kcal_{{\Gamma}}[\varphi](x):=\mbox{p.v.}\int_{{\Gamma}}\frac{\p G(x-y)}{\p \nu_y}\varphi(y)~\mathrm{d} s(y),
\end{equation}
where p.v. stands for the Cauchy principle value.
The single layer potential operator $\Scal_{\Gamma} $ and the double layer potential operator $\mathcal{D}_{\Gamma} $ satisfy the trace formulae (cf. \cite{HK07:book})
\begin{equation} \label{eq:trace}
\frac{\p}{\p\nu}\Scal_{\Gamma} [\varphi] \Big|_{\pm} = (\pm \frac{1}{2}I+
\Kcal_{{\Gamma}}^*)[\varphi] \quad \mbox{on } {\Gamma},
\end{equation}
\[
\mathcal{D}_{{\Gamma}}[\varphi]\Big|_{\pm}=(\mp \frac{1}{2}I+
\Kcal_{{\Gamma}} )[\varphi] \quad \mbox{on } {\Gamma},
\]
where $\Kcal_{{\Gamma}}^*$ is the adjoint operator of $\Kcal_{\Gamma} $ with respect to the $L^2$ inner product.

It can be seen that the solution $u$ to \eqref{eq:mainmd02} may be represented as
\begin{equation}\label{eq:general_solution}
u(x)=H(x)+\sum_{k=1}^{N}\Scal_{\Gamma_k}[\phi_k](x)
\end{equation}
for some functions $\phi_k\in L^2_0(
\Gamma_k)$. Since $\Scal_{\Gamma_k}[\phi_k]$ is continuous across $\Gamma_k$, the first condition in \eqref{eq:transmscds} is automatically satisfied. By using the second condition in \eqref{eq:transmscds}, we can deduce the following equations
\[
\sigma_{k-1}\left(\frac{\p H}{\p \nu_k}+\left.\frac{\p \Scal_{\Gamma_k}[\phi_k]}{\p \nu_k}\right|_{+}+\sum_{l\neq k}^{N}\frac{\p \Scal_{\Gamma_l}[\phi_l]}{\p \nu_k}\right)=\sigma_{k}\left(\frac{\p H}{\p \nu_k}+\left.\frac{\p \Scal_{\Gamma_k}[\phi_k]}{\p \nu_k}\right|_{-}+\sum_{l\neq k}^{N}\frac{\p \Scal_{\Gamma_l}[\phi_l]}{\p \nu_k}\right).
\]
Using the jump formula \eqref{eq:trace} for the normal derivative of the single layer potentials, the above equations can be rewritten as
\begin{equation}\label{eq:integralrep}
\begin{split}
\begin{bmatrix}
\lambda_{1}I - \Kcal_{\Gamma_1}^* & -\nu_1\cdot\nabla\Scal_{\Gamma_2} & \cdots & -\nu_1\cdot\nabla\Scal_{\Gamma_{N}} \\
-\nu_2\cdot\nabla\Scal_{\Gamma_1} & \lambda_{2}I-\Kcal_{\Gamma_2}^* & \cdots & -\nu_2\cdot\nabla\Scal_{\Gamma_{N}}\\
\vdots & \vdots &\ddots &\vdots \\
-\nu_{N}\cdot\nabla\Scal_{\Gamma_1} & -\nu_{N}\cdot\nabla\Scal_{\Gamma_2} & \cdots & \lambda_{N}I-\Kcal_{\Gamma_{N}}^*
\end{bmatrix}
\begin{bmatrix}
\phi_1  \\
\phi_2 \\
\vdots  \\
\phi_N
\end{bmatrix} = \begin{bmatrix}
\nu_1\cdot\nabla H  \\
\nu_2\cdot\nabla H \\
\vdots  \\
\nu_N\cdot\nabla H
\end{bmatrix},
\end{split}
\end{equation}
on $\mathcal{H}_0 = L_0^2(\Gamma_1)\times L_0^2(\Gamma_2)\times\cdots\times L_0^2(\Gamma_{N})$,
where
\begin{equation}\label{lamdk}
\lambda_k=\frac{\sigma_{k}+\sigma_{k-1}}{2(\sigma_{k}-\sigma_{k-1})},\quad k=1,2,\ldots,N,
\end{equation}
and $\sigma_0 = 1$. Here $\sigma_{k}\neq\sigma_{k-1}$, $k=1,2,\ldots,N.$
Let
$\mathbb{K}_A^*$ be an $N$-by-$N$ matrix type NP operator  on $\mathcal{H} :=L^2(\Gamma_1)\times L^2(\Gamma_2)\times\cdots\times L^2(\Gamma_{N})$ defined by
\begin{equation}\label{eq:app11}
\begin{split}
\mathbb{K}_A^*:=
\begin{bmatrix}
\Kcal_{\Gamma_1}^* & \nu_1\cdot\nabla\Scal_{\Gamma_2} & \cdots & \nu_1\cdot\nabla\Scal_{\Gamma_{N}} \\
\nu_2\cdot\nabla\Scal_{\Gamma_1} & \Kcal_{\Gamma_2}^* & \cdots & \nu_2\cdot\nabla\Scal_{\Gamma_{N}}\\
\vdots & \vdots &\ddots &\vdots \\
\nu_{N}\cdot\nabla\Scal_{\Gamma_1} & \nu_{N}\cdot\nabla\Scal_{\Gamma_2} & \cdots & \Kcal_{\Gamma_{N}}^*
\end{bmatrix},
\end{split}
\end{equation}
and let  $\bm\phi := (\phi_1,\phi_2,\ldots,\phi_{N})^T$, $\bm{g}:=\left(\nu_1\cdot\nabla H,\nu_2\cdot\nabla H,\ldots,\nu_N\cdot\nabla H\right)^T$.
Then, \eqref{eq:integralrep} can be rewritten in the form
\begin{equation}\label{IRE}
(\mathbb{I}^{\lambda}-\mathbb{K}_A^*)\bm\phi = \bm{g},
\end{equation}
where $\mathbb{I}^{\lambda}$ is given by
\[
\mathbb{I}^{\lambda}:=
\begin{bmatrix}
\lambda_{1}I & 0 & \cdots & 0\\
0 & \lambda_{2}I & \cdots & 0\\
\vdots & \vdots &\ddots &\vdots \\
0 & 0 & \cdots & \lambda_{N}I
\end{bmatrix}.
\]

For the spectrum of $\mathbb{K}_A^*$, we have the following result which is a generalization of \cite[Lemma 3.1]{ACKLM1} on two-layer structures.
\begin{lem}\label{SPNP}
The spectrum of $\mathbb{K}_A^*$ on $\mathcal{H}$ lies in the interval $(-1/2, 1/2]$.
\end{lem}
\begin{proof}[\bf Proof]
Denote by $\langle u,v \rangle_{L^2(\Gamma)}$ the Hermitian product on $L^{2}(\Gamma)$ with  $\Gamma = \Gamma_k$, for some $k=1,2,\ldots,N.$ By interchange orders of integration, it is easy to see that for $l\neq k $,
\begin{equation} \label{adjointdiSe}
\left\langle \pd{\Scal_{\Gamma_l}[\vp_l]}{\nu_k} , \vp_k
\right\rangle_{L^2(\Gamma_k)} = \left\langle \vp_l, \Dcal_{\Gamma_k}[\vp_k]
\right\rangle_{L^2(\Gamma_l)}.
\end{equation}
Let $\lambda$ be a point in the spectrum of $\mathbb{K}_A^*$.  Then there exists a non-zero vector $\bm\phi = (\phi_1,\phi_2,\ldots,\phi_{N})^T\in \mathcal{H}$ such that
\begin{equation}\label{solu22}
\Kcal_{\Gamma_k}^*[\phi_k]+\sum_{l\neq k}^{N}\frac{\p \Scal_{\Gamma_l}[\phi_l]}{\p \nu_k}=\lambda\phi_k, \quad\mbox{on}\quad \Gamma_k, \quad k=1,2,\ldots,N.
\end{equation}
By integrating the above equations on $\Gamma_k$, $k=1,2,\ldots,N$, and using
\eqref{adjointdiSe}, we obtain
\begin{equation}\label{equ213}
\left\{
\begin{array}{ll}
\left(\lambda-\frac{1}{2}\right)\int_{ \Gamma_{k}}\phi_k(y)~\mathrm{d}s(y)=\sum_{l=k+1}^{N}\int_{ \Gamma_{l}}\phi_l(y)~\mathrm{d}s(y), & k=1,2,\ldots,N-1,\\
\left(\lambda-\frac{1}{2}\right)\int_{ \Gamma_{k}}\phi_k(y)~\mathrm{d}s(y)=0, &k=N.
\end{array}
\right.
\end{equation}
Here, we used the facts that $\Kcal_{\Gamma_k}[1]={1}/{2}$, for all
$k=1,2,\ldots,N,$ and
\[
\left.\Dcal_{\Gamma_k}[1]\right|_{\Gamma_l} =\left\{
\begin{array}{ll}
1, & l>k,\\
0, &l<k.
\end{array}
\right.
\]
Thus, from \eqref{equ213}, we have that either $\lambda = 1/2$ or $\lambda \neq 1/2$ with $\phi_k\in L_0^2(\Gamma_k)$, for all $k=1,2,\ldots,N,$ holds.
We next assume that $\lambda \neq 1/2$ and consider
\[
u(x):=\sum_{k=1}^{N}\Scal_{\Gamma_k}[\phi_k](x), \quad x\in \RR^d
\]
for $d\geqslant 2$.
Since $\phi_k\in L_0^2(\Gamma_k)$, $k=1,2,\ldots,N,$ we have $
u(x) = O(|x|^{1-d})$, and  $\nabla u(x) = O(|x|^{-d}),
$
as $|x|\to\infty$ for $d\geqslant 2$.
Hence the following integrals are finite:
\begin{equation}\label{equ214}
V_k:=\int_{A_{k}}|\nabla u |^2~\mathrm{d}x\geqslant 0, \quad k=0,1,\ldots,N.
\end{equation}
We next claim
\begin{equation}\label{sumge0}
\sum_{k=0}^{N}V_k>0.
\end{equation}
Indeed, if $V_k = 0$ for all $k=0,1,\ldots,N,$ then $u(x)=$ constant in $A_k$ for all $k=0,1,\ldots,N.$ It follows that
\[
\phi_k=\left.\frac{\p u}{\p \nu_k}\right|_+-\left.\frac{\p u}{\p \nu_k}\right|_-=0, \mbox{ for all } k=1,2,\ldots,N.
\]
Hence $\bm\phi = \bm 0$, which is a contradiction.

On the other hand,  we obtain from Green's formulas, the jump relation \eqref{eq:trace}, and \eqref{solu22} that
\begin{equation}\label{equ215}
\left\{
\begin{array}{ll}
V_0 = -\(\lambda+\frac{1}{2}\)\int_{ \Gamma_{1}}\phi_1u~\mathrm{d}s, & \\
V_k =\(\lambda-\frac{1}{2}\)\int_{ \Gamma_{k}}\phi_ku~\mathrm{d}s -\(\lambda+\frac{1}{2}\)\int_{ \Gamma_{k+1}}\phi_{k+1}u~\mathrm{d}s, &k=1,2,\ldots,N-1,\\
V_N = \(\lambda-\frac{1}{2}\)\int_{ \Gamma_{N}}\phi_Nu~\mathrm{d}s.&
\end{array}
\right.
\end{equation}
It follows that
\begin{equation}\label{equ216}
\lambda = \frac{V_0-\sum_{k=1}^{N}V_k}{2\left(\sum_{k=0}^{N}V_k\right)}.
\end{equation}
It follows from \eqref{equ214} and \eqref{sumge0} that $-1/2<\lambda<1/2$.

The proof is complete.
\end{proof}

Based on the analysis above, we are now in the position to present the integral representation for the perturbation filed.
\begin{thm}\label{repsolunq}
	Let $u$ be the solution of the conductivity problem \eqref{eq:mainmd02} in $\RR^d$ for $d = 2\mbox{ or } 3$, with the conductivity $\sigma $ given by \eqref{eq:paracho01} and the transmission conditions given by \eqref{eq:transmscds}.
	There are unique functions $\phi_k\in L_0^2(\Gamma_k)$, $k=1,2,\ldots,N,$ such that
	\begin{equation}\label{rep_solu}
	u(x)=H(x)+\sum_{k=1}^{N}\Scal_{\Gamma_k}[\phi_k](x).
	\end{equation}
	The potentials $\phi_k, k=1,2,\ldots,N,$ satisfy
	\begin{equation}\label{solu2}
	\left(\lambda_k-\Kcal_{\Gamma_k}^*\right)[\phi_k]-\sum_{l\neq k}^{N}\left.\frac{\p \Scal_{\Gamma_l}[\phi_l]}{\p \nu_k}\right|_{\Gamma_k}=\left.\frac{\p H}{\p \nu_k}\right|_{\Gamma_k}.
	\end{equation}
\end{thm}
\begin{proof}[\bf Proof]
	It follows from \eqref{eq:trace} that $u$ defined by \eqref{rep_solu} and \eqref{solu2} is the solution of  the transmission problem \eqref{eq:mainmd02}--\eqref{eq:transmscds}. Then it suffices to prove that the integral equation \eqref{solu2} has a unique solution.
	
	We next prove that the operator $T:\mathcal{H}_0\to\mathcal{H}_0$ defined by
	\[
	\begin{aligned}
	T(\phi_1,\phi_2,\ldots,\phi_{N})&=T_0(\phi_1,\phi_2,\ldots,\phi_{N})+T_1(\phi_1,\phi_2,\ldots,\phi_{N})\\
	&:=\((\lambda_1-\Kcal_{\Gamma_1}^*)[\phi_1],(\lambda_2-\Kcal_{\Gamma_2}^*)[\phi_2],\ldots,(\lambda_N-\Kcal_{\Gamma_N}^*)[\phi_N]\)\\
	&\quad - \(\sum_{l\neq 1}^{N}\left.\frac{\p \Scal_{\Gamma_l}[\phi_l]}{\p \nu_1}\right|_{\Gamma_1},\sum_{l\neq 2}^{N}\left.\frac{\p \Scal_{\Gamma_l}[\phi_l]}{\p \nu_2}\right|_{\Gamma_2},\ldots,\sum_{l\neq N}^{N}\left.\frac{\p \Scal_{\Gamma_l}[\phi_l]}{\p \nu_N}\right|_{\Gamma_N}\)
	\end{aligned}
	\]
	is invertible. From \cite[Theorem 2.21]{HK07:book}, one has that $T_0$ is invertible on $\mathcal{H}_0$.  Moreover, due to the fact that the surfaces $\Gamma_l$ do not intersect,  then $T_1$ is compact on $\mathcal{H}_0$. Therefore, by the Fredholm alternative, it suffices to prove that $T$ is injective  on $\mathcal{H}_0$.
	If $T(\phi_1,\phi_2,\ldots,\phi_{N})=0,$ then
	\[
	u(x)=\sum_{k=1}^{N}\Scal_{\Gamma_k}[\phi_k](x)
	\]
	is the solution  to \eqref{eq:mainmd02} with $H=0$. By the well-posedness of \eqref{eq:mainmd02}--\eqref{eq:transmscds}, we get $u\equiv 0$. Particularly, $\Scal_{\Gamma_k}[\phi_k]$ is smooth across $\Gamma_k$, $k=1,2,\ldots,N$.
	Hence,
	\[
	\phi_k=\left.\frac{\p \Scal_{\Gamma_k}[\phi_k]}{\p \nu_k}\right|_+-\left.\frac{\p \Scal_{\Gamma_k}[\phi_k]}{\p \nu_k}\right|_-=0.
	\]
	The proof is complete.
\end{proof}

\section{Generalized Polarization Tensors of multi-layer structures}\label{sec:03}
Our aim in this section is to introduce the concept of Generalized Polarization Tensors of multi-layer structures. These concepts are defined in a way analogous to the
generalized polarization tensors introduced in \cite{ADKL14,HK04book}. We also give some important properties for the GPTs.
These results will turn out to be crucial for our approach to determine the location
and some geometric and material features of multi-layer structures.

\subsection{Definition of GPTs}
With Theorem \ref{repsolunq}, we can proceed to introduce the polarization
tensors of multi-layer structures.
For a multi-index $\alpha=(\alpha_1,\ldots,\alpha_d)\in\mathbb{N}^d$, let $x^{\alpha} = x^{\alpha_1}_1\cdots x^{\alpha_d}_d $ and $\partial^{\alpha}=\partial^{\alpha_1}_1\cdots \partial^{\alpha_d}_d $, with $\partial_j=\partial/\partial x_j$. Denote by $\bm{e}_k:=(0,0,\ldots,1,0,\ldots,0)^T$ the $N$-dimensional vector with the $k$-th entrance be one.
With the help of Lemma \ref{SPNP} and \eqref{IRE}, we have that
\[
(u-H)(x) = \sum_{k=1}^{N}\Scal_{\Gamma_k}(\bm{e}_k^T(\mathbb{I}^{\lambda}-\mathbb{K}_A^*)^{-1}\left(\left(\nu_1\cdot\nabla H,\nu_2\cdot\nabla H,\ldots,\nu_N\cdot\nabla H\right)^T)\right)(x),
\]
this, together with the Taylor expansion
\[
G(x-y) = \sum_{|\alpha|=0}^{+\infty}\frac{(-1)^\alpha}{\alpha!}\partial^{\alpha} G(x)y^{\alpha},\quad x\to+\infty,
\]
and $y$ in a compact set, we can obtain that the far-field expansion for the  perturbed electric potential
\begin{equation}\label{FFE}
\begin{aligned}
&\quad (u-H)(x)\\&= \sum_{k=1}^{N}\int_{ \Gamma_{k}}G(x-y)(\bm{e}_k^T(\mathbb{I}^{\lambda}-\mathbb{K}_A^*)^{-1}\left(\left(\nu_1\cdot\nabla H,\nu_2\cdot\nabla H,\ldots,\nu_N\cdot\nabla H\right)^T\right)~\mathrm{d} s(y)\\
& = \sum_{k=1}^{N}\sum^{+\infty}_{|\alpha|=1} \sum^{+\infty}_{|\beta|=1}\frac{(-1)^{|\alpha|}}{\alpha!\beta!}\partial^{\alpha} G(x)\partial^{\beta} H(0)\int_{ \Gamma_{k}}y^{\alpha}(\bm{e}_k^T(\mathbb{I}^{\lambda}-\mathbb{K}_A^*)^{-1}\left(\left(\nu_1\cdot\nabla y^{\beta},\nu_2\cdot\nabla y^{\beta},\ldots,\nu_N\cdot\nabla y^{\beta}\right)^T\right)~\mathrm{d} s(y),
\end{aligned}
\end{equation}
as $x\to+\infty,$ where $(\bm{e}_1,\bm{e}_2,\ldots,\bm{e}_N)$ is an orthonormal basis of $\RR^N$.

\begin{defn}
	For $\alpha$, $\beta\in \mathbb{N}^d$, let $\phi_{k,\beta}$, $k=1,2,\ldots,N$, be the solution of
	\begin{equation}\label{solu3}
	\left(\lambda_k-\Kcal_{\Gamma_k}^*\right)[\phi_{k,\beta}]-\sum_{l\neq k}^{N}\left.\frac{\p \Scal_{\Gamma_l}[\phi_{l,\beta}]}{\p \nu_k}\right|_{\Gamma_k}=\left.\frac{\p y^{\beta}}{\p \nu_k}\right|_{\Gamma_k}.
	\end{equation}
	Then the generalized polarization tensor (GPT) $M_{\alpha\beta}$ is defined to be
	\begin{equation}\label{eq:def_M2}
	M_{\alpha\beta} := \sum_{k=1}^{N}\int_{\Gamma_k} {y}^{\alpha} {\phi}_{k,\beta}({y})~\mathrm{d}s({y}).
	\end{equation}
	If $|\alpha| = |\beta|=1$, we denote $M_{\alpha\beta}$ by $M_{ij}$, $i,j = 1,\ldots,d,$ and call $\mathbf{M} =(M_{ij})_{i,j=1}^{d} $ first-order polarization tensor.
\end{defn}

Formula \eqref{FFE} shows that through the GPTs we have complete information
about the far-field expansion of perturbed electric potential
\begin{equation}\label{FFEPE}
\begin{aligned}
\quad (u-H)(x) = \sum^{+\infty}_{|\alpha|=1} \sum^{+\infty}_{|\beta|=1}\frac{(-1)^{|\alpha|}}{\alpha!\beta!}\partial^{\alpha} G(x)M_{\alpha\beta}\partial^{\beta} H(0),\quad \mbox{ as } x\to+\infty.
\end{aligned}
\end{equation}

\subsection{Properties of GPTs}
In this subsection, we study some interesting physical properties of GPTs, such as   symmetry and positivity. We
emphasize that the harmonic sums of GPTs play a key role. Let $I$ and $J$ be finite index sets. Harmonic
sums of GPTs are $\sum_{\alpha \in I, \beta \in J} a_{\alpha}
b_{\beta} M_{\alpha\beta}$ where $\sum_{\alpha \in I} a_{\alpha}
x^{\alpha}$ and $\sum_{\beta \in J} b_{\beta} x^{\beta}$ are
harmonic polynomials.

We shall derive some symmetry property of the GPTs in the following theorem. Due to the consideration of the special case of inhomogeneous medium: multi-layered  medium,  our definition of GPTs \eqref{eq:def_M2} for multi-layered medium in this paper is more refined compared to the definition in \cite{ADKL14}. We mention that the following theorem can be proved by using the similar arguments in \cite[Lemma 4.2]{ADKL14} and the GPTs therein. Nevertheless, we shall reformulate the proof by using the new form of GPTs defined in \eqref{eq:def_M2}.

\begin{thm}\label{th:sym} Let $I$ and $J$ be finite index sets.
	For any harmonic coefficients $\{a_{\alpha}|\alpha \in I\}$ and
	$\{b_{\beta}|\beta \in J\}$,  we have
	\begin{equation}\label{eq:107}
	\sum_{\alpha \in I}\sum_{\beta \in J} a_{\alpha} b_{\beta} M_{\alpha\beta} = \sum_{\alpha \in I}\sum_{\beta \in J} a_{\alpha} b_{\beta} M_{\beta\alpha}.
	\end{equation}
\end{thm}
\begin{proof}[\bf Proof]
Note that
\[
\sum_{\alpha \in I}\sum_{\beta \in J} a_{\alpha} b_{\beta} M_{\alpha\beta} =\sum_{k=1}^{N}\int_{\Gamma_k}\sum_{\alpha \in I} a_{\alpha}{y}^{\alpha}\sum_{\beta \in J} b_{\beta}  {\phi}_{k,\beta}({y})~\mathrm{d}s({y}).
\]
Taking
\[
f(y) = \sum_{\alpha \in I} a_{\alpha}y^{\alpha}, \quad h(y) = \sum_{\beta \in J} b_{\beta}y^{\beta},
\]
\[
\phi_k(y) =\sum_{\alpha \in I} a_{\alpha}\phi_{k,\alpha}(y)\quad\mbox{ and }\quad \psi_k(y) = \sum_{\beta \in J} b_{\beta}\phi_{k,\beta}(y),
\]
it is easy to see that
\[
\sum_{\alpha \in I}\sum_{\beta \in J} a_{\alpha} b_{\beta} M_{\alpha\beta} =\sum_{k=1}^{N}\int_{\Gamma_k}f(y)\psi_k(y)~\mathrm{d}s({y}),
\]
and
\[
\quad \sum_{\alpha \in I}\sum_{\beta \in J} a_{\alpha} b_{\beta} M_{\beta\alpha} = \sum_{k=1}^{N}\int_{\Gamma_k}h(y)\phi_k(y)~\mathrm{d}s({y}).
\]
We next define
\begin{equation}\label{PHSI}
\Phi(x):=\sum_{k=1}^{N}\Scal_{\Gamma_k}[\phi_k](x)\quad\mbox{and}\quad\Psi(x):=\sum_{k=1}^{N}\Scal_{\Gamma_k}[\psi_k](x).
\end{equation}
From the definition of $\phi_{k,\beta}$ , one can readily obtain
\begin{equation}\label{eq:hPsi}
\sigma_{k-1}\frac{\p (h+\Psi)}{\p \nu_k}|_{+}=\sigma_{k}\frac{\p (h+\Psi)}{\p \nu_k}|_{-}\quad\mbox{on}\quad \Gamma_k, \quad k=1,2,\ldots,N,
\end{equation}
and the same relation for $f + \Phi$ holds. From \eqref{solu3}, we get that on $\Gamma_k,$ $ k=1,2,\ldots,N,$
\[
\begin{aligned}
\sigma_{k-1}\left.\frac{\p (\Scal_{\Gamma_k}[\psi_k])}{\p \nu_k}\right|_{+}-\sigma_{k}\left.\frac{\p (\Scal_{\Gamma_k}[\psi_k])}{\p \nu_k}\right|_{-} &=\sum_{\beta \in J} b_{\beta} \left(\sigma_{k-1}\left.\frac{\p (\Scal_{\Gamma_k}[\phi_{k,\beta}])}{\p \nu_k}\right|_{+}-\sigma_{k}\left.\frac{\p (\Scal_{\Gamma_k}[\phi_{k,\beta}])}{\p \nu_k}\right|_{-}\right)\\
&=(\sigma_{k}-\sigma_{k-1})\sum_{\beta \in J} b_{\beta}\frac{\p}{\p \nu_k}\left( y^{\beta}+\sum_{l\neq k}^{N}{ \Scal_{\Gamma_l}[\phi_{l,\beta}]}\right)\\
&=(\sigma_{k}-\sigma_{k-1})\frac{\p}{\p \nu_k}\left( h+\sum_{l\neq k}^{N}{ \Scal_{\Gamma_l}[\psi_{l}]}\right).
\end{aligned}
\]
Thus, it follows from \eqref{eq:hPsi} that
\begin{equation}\label{psik}
\begin{aligned}
\psi_k&=\left.\frac{\p \Scal_{\Gamma_k}[\psi_k]}{\p \nu_k}\right|_+-\left.\frac{\p \Scal_{\Gamma_k}[\psi_k]}{\p \nu_k}\right|_-
\\
&=\left.\frac{\p (\Scal_{\Gamma_k}[\psi_k])}{\p \nu_k}\right|_{+}-\frac{\sigma_{k}}{\sigma_{k-1}}\left.\frac{\p (\Scal_{\Gamma_k}[\psi_k])}{\p \nu_k}\right|_{-} + \left(\frac{\sigma_{k}}{\sigma_{k-1}}-1\right)\left.\frac{\p (\Scal_{\Gamma_k}[\psi_k])}{\p \nu_k}\right|_{-}\\
& = \left(\frac{\sigma_{k}}{\sigma_{k-1}}-1\right)\frac{\p}{\p \nu_k}\left( h+\sum_{l\neq k}^{N}{ \Scal_{\Gamma_l}[\psi_{l}]}\right) + \left(\frac{\sigma_{k}}{\sigma_{k-1}}-1\right)\left.\frac{\p (\Scal_{\Gamma_k}[\psi_k])}{\p \nu_k}\right|_{-}\\
& =  \left(\frac{\sigma_{k}}{\sigma_{k-1}}-1\right)\left.\frac{\p (h+\Psi)}{\p \nu_k}\right|_{-}.
\end{aligned}
\end{equation}
Therefore, we get
\begin{equation}\label{SYMGPT}
\begin{aligned}
\sum_{\alpha \in I}\sum_{\beta \in J} a_{\alpha} b_{\beta} M_{\alpha\beta} &=\sum_{k=1}^{N}\left(\frac{\sigma_{k}}{\sigma_{k-1}}-1\right)\int_{\Gamma_k}f\left.\frac{\p (h+\Psi)}{\p \nu_k}\right|_{-}~\mathrm{d}s({y})\\
& = \sum_{k=1}^{N}\left(\frac{\sigma_{k}}{\sigma_{k-1}}-1\right)\int_{\Gamma_k}(f+\Phi)\left.\frac{\p (h+\Psi)}{\p \nu_k}\right|_{-}~\mathrm{d}s({y}) -\sum_{k=1}^{N}\left(\frac{\sigma_{k}}{\sigma_{k-1}}-1\right)\int_{\Gamma_k}\Phi\left.\frac{\p (h+\Psi)}{\p \nu_k}\right|_{-}~\mathrm{d}s({y})\\
& = \sum_{k=1}^{N}\left(\frac{1}{\sigma_{k-1}}-\frac{1}{\sigma_{k}}\right)\sigma_{k}\int_{\Gamma_k}(f+\Phi)\left.\frac{\p (h+\Psi)}{\p \nu_k}\right|_{-}~\mathrm{d}s({y})\\
&\quad -\sum_{k=1}^{N}\int_{\Gamma_k}\Phi\left(\left.\frac{\p \Scal_{\Gamma_k}[\psi_k]}{\p \nu_k}\right|_+-\left.\frac{\p \Scal_{\Gamma_k}[\psi_k]}{\p \nu_k}\right|_-\right)~\mathrm{d}s({y})\\
& = \sum_{k=1}^{N}\left(\frac{1}{\sigma_{k-1}}-\frac{1}{\sigma_{k}}\right)\sigma_{k}\int_{\Gamma_k}(f+\Phi)\left.\frac{\p (h+\Psi)}{\p \nu_k}\right|_{-}~\mathrm{d}s({y})\\
&\quad -\sum_{k=1}^{N}\int_{\Gamma_k}\Phi\left.\frac{\p \Scal_{\Gamma_k}[\psi_k]}{\p \nu_k}\right|_+~\mathrm{d}s({y})
+\sum_{k=1}^{N}\int_{\Gamma_k}\Phi\left.\frac{\p \Scal_{\Gamma_k}[\psi_k]}{\p \nu_k}\right|_-~\mathrm{d}s({y}).
\end{aligned}
\end{equation}
We next analyze \eqref{SYMGPT} term by term. For convenience we use the the  notation
$
\langle u,v \rangle_D = \int_{D}\nabla u\cdot \nabla v ~\mathrm{d}x,
$
where $D$ is a Lipschitz domain in $\mathbb{R}^d$.
It follows from \eqref{eq:hPsi} that
\[
\begin{aligned}
&\quad\sum_{k=1}^{N}\left(\frac{1}{\sigma_{k-1}}-\frac{1}{\sigma_{k}}\right)\sigma_{k}\int_{\Gamma_k}(f+\Phi)\left.\frac{\p (h+\Psi)}{\p \nu_k}\right|_{-}~\mathrm{d}s({y})\\
&=\sum_{k=1}^{N-1}\left(\frac{1}{\sigma_{k-1}}-\frac{1}{\sigma_{k}}\right)\sigma_{k}\int_{\Gamma_{k+1}}(f+\Phi)\left.\frac{\p (h+\Psi)}{\p \nu_{k+1}}\right|_+~\mathrm{d}s({y})+\sum_{k=1}^{N}\left(\frac{1}{\sigma_{k-1}}-\frac{1}{\sigma_{k}}\right)\sigma_{k} \langle f+\Phi,h+\Psi \rangle_{A_k}\\
&= \sum_{k=1}^{N-1}\left(\frac{1}{\sigma_{k-1}}-\frac{1}{\sigma_{k}}\right){\sigma_{k+1}}\int_{\Gamma_{k+1}}(f+\Phi)\left.\frac{\p (h+\Psi)}{\p \nu_{k+1}}\right|_-~\mathrm{d}s({y})+\sum_{k=1}^{N}\left(\frac{1}{\sigma_{k-1}}-\frac{1}{\sigma_{k}}\right)\sigma_{k} \langle f+\Phi,h+\Psi \rangle_{A_k}\\
&= \sum_{k=1}^{N-2}\left(\frac{1}{\sigma_{k-1}}-\frac{1}{\sigma_{k}}\right){\sigma_{k+1}}\int_{\Gamma_{k+2}}(f+\Phi)\left.\frac{\p (h+\Psi)}{\p \nu_{k+2}}\right|_+~\mathrm{d}s({y})\\
&\quad +\sum_{m=N-1}^{N}\sum_{k=1}^{m}\left(\frac{1}{\sigma_{k-1}}-\frac{1}{\sigma_{k}}\right)\sigma_{k+N-m} \langle f+\Phi,h+\Psi \rangle_{A_{k+N-m}}\\
&= \left(\frac{1}{\sigma_{0}}-\frac{1}{\sigma_{1}}\right){\sigma_{N}}\int_{\Gamma_{N}}(f+\Phi)\left.\frac{\p (h+\Psi)}{\p \nu_{N}}\right|_-~\mathrm{d}s({y})+\sum_{m=2}^{N}\sum_{k=1}^{m}\left(\frac{1}{\sigma_{k-1}}-\frac{1}{\sigma_{k}}\right)\sigma_{k+N-m} \langle f+\Phi,h+\Psi \rangle_{A_{k+N-m}}\\
&= \sum_{m=1}^{N}\sum_{k=1}^{m}\left(\frac{1}{\sigma_{k-1}}-\frac{1}{\sigma_{k}}\right)\sigma_{k+N-m} \langle f+\Phi,h+\Psi \rangle_{A_{k+N-m}}.\\
\end{aligned}
\]
Then by direct calculation, one further has that
\[
\sum_{m=1}^{N}\sum_{k=1}^{m}\left(\frac{1}{\sigma_{k-1}}-\frac{1}{\sigma_{k}}\right)\sigma_{k+N-m} \langle f+\Phi,h+\Psi \rangle_{A_{k+N-m}} = \sum_{k=1}^{N}\left({\sigma_{k}}-1\right) \langle f+\Phi,h+\Psi \rangle_{A_{k}}.
\]
In a similar manner, one can show that
\[
\begin{aligned}
&\quad \sum_{k=1}^{N}\int_{\Gamma_k}\Phi\left.\frac{\p \Scal_{\Gamma_k}[\psi_k]}{\p \nu_k}\right|_+~\mathrm{d}s({y}) = \sum_{l=1}^{N}\sum_{k=1}^{N}\int_{\Gamma_k}\Scal_{\Gamma_l}[\phi_l]\left.\frac{\p \Scal_{\Gamma_k}[\psi_k]}{\p \nu_k}\right|_+~\mathrm{d}s({y})\\
& = \sum_{l=1}^{N}\left(\sum_{k=2}^{N}\int_{\Gamma_{k-1}}\Scal_{\Gamma_l}[\phi_l]\left.\frac{\p \Scal_{\Gamma_k}[\psi_k]}{\p \nu_{k-1}}\right|_-~\mathrm{d}s({y})-\sum_{k=1}^{N}\langle \Scal_{\Gamma_l}[\phi_l],{ \Scal_{\Gamma_k}[\psi_k]} \rangle_{A_{k-1}}\right)\\
& = \sum_{l=1}^{N}\left(\sum_{k=2}^{N}\int_{\Gamma_{k-1}}\Scal_{\Gamma_l}[\phi_l]\left.\frac{\p \Scal_{\Gamma_k}[\psi_k]}{\p \nu_{k-1}}\right|_+~\mathrm{d}s({y})-\sum_{k=1}^{N}\langle \Scal_{\Gamma_l}[\phi_l],{ \Scal_{\Gamma_k}[\psi_k]} \rangle_{A_{k-1}}\right)\\
& = \sum_{l=1}^{N}\left(\sum_{k=3}^{N}\int_{\Gamma_{k-2}}\Scal_{\Gamma_l}[\phi_l]\left.\frac{\p \Scal_{\Gamma_k}[\psi_k]}{\p \nu_{k-2}}\right|_-~\mathrm{d}s({y})-\sum_{m=1}^{2}\sum_{k=m}^{N}\langle \Scal_{\Gamma_l}[\phi_l],{ \Scal_{\Gamma_k}[\psi_k]} \rangle_{A_{k-m}}\right)\\
&= \sum_{l=1}^{N}\left(\int_{\Gamma_{1}}\Scal_{\Gamma_l}[\phi_l]\left.\frac{\p \Scal_{\Gamma_N}[\psi_N]}{\p \nu_{1}}\right|_+~\mathrm{d}s({y})-\sum_{m=1}^{N-1}\sum_{k=m}^{N}\langle \Scal_{\Gamma_l}[\phi_l],{ \Scal_{\Gamma_k}[\psi_k]} \rangle_{A_{k-m}}\right)\\
& = -\sum_{l=1}^{N}\sum_{m=1}^{N}\sum_{k=m}^{N}\langle \Scal_{\Gamma_l}[\phi_l],{ \Scal_{\Gamma_k}[\psi_k]} \rangle_{A_{k-m}},
\end{aligned}
\]
and
\[
\begin{aligned}
&\quad\sum_{k=1}^{N}\int_{\Gamma_k}\Phi\left.\frac{\p \Scal_{\Gamma_k}[\psi_k]}{\p \nu_k}\right|_-~\mathrm{d}s({y})= \sum_{l=1}^{N}\sum_{k=1}^{N}\int_{\Gamma_k}\Scal_{\Gamma_l}[\phi_l]\left.\frac{\p \Scal_{\Gamma_k}[\psi_k]}{\p \nu_k}\right|_-~\mathrm{d}s({y})\\
& = \sum_{l=1}^{N}\left(\sum_{k=1}^{N-1}\int_{\Gamma_{k+1}}\Scal_{\Gamma_l}[\phi_l]\left.\frac{\p \Scal_{\Gamma_k}[\psi_k]}{\p \nu_{k+1}}\right|_+~\mathrm{d}s({y})+\sum_{k=1}^{N}\langle \Scal_{\Gamma_l}[\phi_l],{ \Scal_{\Gamma_k}[\psi_k]} \rangle_{A_{k}}\right)\\
& = \sum_{l=1}^{N}\left(\sum_{k=1}^{N-1}\int_{\Gamma_{k+1}}\Scal_{\Gamma_l}[\phi_l]\left.\frac{\p \Scal_{\Gamma_k}[\psi_k]}{\p \nu_{k+1}}\right|_-~\mathrm{d}s({y})+\sum_{k=1}^{N}\langle \Scal_{\Gamma_l}[\phi_l],{ \Scal_{\Gamma_k}[\psi_k]} \rangle_{A_{k}}\right)\\
& = \sum_{l=1}^{N}\left(\sum_{k=1}^{N-2}\int_{\Gamma_{k+2}}\Scal_{\Gamma_l}[\phi_l]\left.\frac{\p \Scal_{\Gamma_k}[\psi_k]}{\p \nu_{k+2}}\right|_+~\mathrm{d}s({y})+\sum_{m=N-1}^{N}\sum_{k=1}^{m}\langle \Scal_{\Gamma_l}[\phi_l],{ \Scal_{\Gamma_k}[\psi_k]} \rangle_{A_{k+N-m}}\right)\\
&= \sum_{l=1}^{N}\left(\int_{\Gamma_{N}}\Scal_{\Gamma_l}[\phi_l]\left.\frac{\p \Scal_{\Gamma_k}[\psi_k]}{\p \nu_{N}}\right|_-~\mathrm{d}s({y})+\sum_{m=2}^{N}\sum_{k=1}^{m}\langle \Scal_{\Gamma_l}[\phi_l],{ \Scal_{\Gamma_k}[\psi_k]} \rangle_{A_{k+N-m}}\right)\\
& = \sum_{l=1}^{N}\sum_{m=1}^{N}\sum_{k=1}^{m}\langle \Scal_{\Gamma_l}[\phi_l],{ \Scal_{\Gamma_k}[\psi_k]} \rangle_{A_{k+N-m}}.
\end{aligned}
\]
Then we finally obtain
\begin{equation}\label{eq:315}
\begin{aligned}
&\quad\sum_{\alpha \in I}\sum_{\beta \in J} a_{\alpha} b_{\beta} M_{\alpha\beta} \\
&= \sum_{k=1}^{N}\left({\sigma_{k}}-1\right) \langle f+\Phi,h+\Psi \rangle_{A_{k}} \\
&\quad  + \sum_{l=1}^{N}\sum_{m=1}^{N}\sum_{k=m}^{N}\langle \Scal_{\Gamma_l}[\phi_l],{ \Scal_{\Gamma_k}[\psi_k]} \rangle_{A_{k-m}} +  \sum_{l=1}^{N}\sum_{m=1}^{N}\sum_{k=1}^{m}\langle \Scal_{\Gamma_l}[\phi_l],{ \Scal_{\Gamma_k}[\psi_k]} \rangle_{A_{k+N-m}}\\
&= \sum_{k=1}^{N}\left({\sigma_{k}}-1\right) \langle f+\Phi,h+\Psi \rangle_{A_{k}}  + \sum_{l=1}^{N}\sum_{k=1}^{N}\langle \Scal_{\Gamma_l}[\phi_l],{ \Scal_{\Gamma_k}[\psi_k]} \rangle_{\mathbb{R}^d} .
\end{aligned}
\end{equation}
The symmetry of \eqref{eq:107} follows immediately from \eqref{eq:315} and the proof is complete.
\end{proof}
In order to give the positivity of GPTs, we have the following bounds for GPTs by
following the similar arguments of proof as in \cite[Theorem 4.1]{ADKL14} for the
general inhomogeneous inclusion.
\begin{thm}\label{Positivedef}
	Let $I$ be a finite index set. Let $\{a_{\alpha} | \alpha \in I\}$
	be the set of coefficients such that $f(x):=\sum_{\alpha \in I}
	a_{\alpha} x^{\alpha}$ is a harmonic function. Then we have
\begin{equation}
	\label{eq:111} \sum_{k=1}^{N}\frac{(\Gs_k-1)}{\Gs_k}\int_{A_k} |\nabla f|^2~\mathrm{d} x \leqslant
	\sum_{\alpha,\beta \in I} a_{\alpha}a_{\beta} M_{\alpha\beta} \leqslant  \sum_{k=1}^{N}\int_{A_k}(\Gs_k-1) |\nabla f|^2 ~\mathrm{d} x.
\end{equation}
\end{thm}
The above theorem shows that if $\Gs_k-1>0$ for all $k =1,2,\ldots,N,$
then the GPTs are positive-definite, and they are negative-definite if $0<\Gs_k<1$ for all $k =1,2,\ldots,N.$

\section{Identification of location  for multi-layer structures}\label{sec2}

In this section, we shall consider the uniqueness in determining the location of multi-layer structures.
Let $A=\cup_{k=1}^NA_k$ denote the multi-layer structure that we are concerned with. It is assumed that $A$ is of the form
\begin{equation}\label{eq:form1}
A=  B + z,
\end{equation}
where $z\in\mathbb{R}^d$, $d=2$ or 3, and $B$ is a bounded domain containing the origin with  a $C^{1,\eta}$ smooth boundary $\widetilde{\Gamma}_1$, and  $B_0=\mathbb{R}^d\backslash\overline{B}$. The interior of $B$ is divided by means of closed and nonintersecting  $C^{1,\eta}$ surfaces $\widetilde{\Gamma}_k$ $(k = 2, 3,...,N)$  into subsets (layers)  $B_k$ $(k = 1, 2,...,N)$.  Each $\widetilde{\Gamma}_{k-1}$ surrounds $\widetilde{\Gamma}_k$ $(k=2,3,\ldots,N)$.  The regions $B_k$ $(k=1,2,\ldots,N)$ are homogeneous media.
Since $A= B + z$, for any $y\in \Gamma_i$, we let $\tilde{y}=(y-z)\in \widetilde{\Gamma}_i, i=1,2,\ldots,N$. Denote by $\widetilde{\varphi}(\widetilde{y})=\varphi(y)$ and $\widetilde{\psi}(\widetilde{y})=\psi(y)$, and let $\partial/\partial\widetilde{\nu_i}$ be the normal derivative on the boundary $\widetilde{\Gamma}_i$.

\begin{lem}
  Let $\phi_k\in L^2(\Gamma_k)$, $k=1,2,\ldots,N$. There hold
\begin{equation}\label{npae}
\Kcal^*_{\Gamma_k}[\phi_k](x)=\Kcal^*_{\widetilde{\Gamma}_k}[\widetilde\phi_k](\widetilde{x}),
\end{equation}	
and
\begin{equation}\label{dsae}
\frac{\p \Scal_{\Gamma_l}[\phi_l]}{\p \nu_k}=\frac{\p \Scal_{\widetilde{\Gamma}_l}[\widetilde\phi_l]}{\p \widetilde{\nu_k}}, \quad
\mbox{for}\quad l\neq k.
\end{equation}
\end{lem}
\begin{proof}[\bf Proof]
Let $x\in \Gamma_{k}$ and denote $\widetilde{x}=(x-z)$.
By using $y= \widetilde{y}+z$ and  change of variables in integrals, one has that
\[
\begin{aligned}
\Kcal^*_{\Gamma_k}[\phi_k](x
)&=\int_{\Gamma_k}\frac{\p G(x-y)}{\p \nu_x}\phi_k(y)~\mathrm{d} s(y)=\nu_x\cdot \nabla_x\int_{\Gamma_k}G(x-y)\phi_k(y)~\mathrm{d} s(y)\\
&=\nu_{\widetilde{x}}\cdot\nabla_{\widetilde{x}}\int_{\widetilde{\Gamma}_k}G( \widetilde{x}- \widetilde{y})\widetilde\phi_k(\widetilde y)~\mathrm{d} s(\widetilde y)\\
&=\Kcal^*_{\widetilde{\Gamma}_k}[\widetilde\phi_k](\widetilde x).
\end{aligned}
\]
Moreover, \eqref{dsae} can be proved in a similar manner. The proof is complete.
\end{proof}
Next,
by Taylor series expansion, the background field $H(y)$ has the following expansion
\begin{equation}\label{eq:direxpF01}
H(y)=H( \widetilde{y} + z) = H(z)+ \sum^{+\infty}_{|\beta|=1} \frac{1}{\beta!}  \widetilde{y}^{\beta} \partial^{\beta}H(z).
\end{equation}
Let $\widetilde\Phi_\beta=\left(\widetilde{\phi}_{1,\beta}, \widetilde{\phi}_{2,\beta},\ldots,\widetilde{\phi}_{N,\beta}\right)$ be the solution to the following equation
\[
\mathbb{J}_B^{\lambda}[\widetilde\Phi_\beta ]=
\left(\frac{\p}{\p\widetilde{\nu_1}}\widetilde{y}^{\beta},
\frac{\p}{\p\widetilde{\nu_2}}\widetilde{y}^{\beta},\ldots,\frac{\p}{\p\widetilde\nu_{N}}\widetilde{y}^{\beta}\right)^T,
\]
where
\begin{equation}\label{main_mat}
\begin{split}
\mathbb{J}_B^{\lambda}:=
\begin{bmatrix}
\lambda_{1}-\Kcal_{\widetilde{\Gamma}_1}^* & -\widetilde{\nu_1}\cdot\nabla\Scal_{\widetilde{\Gamma}_2} & \cdots & -\widetilde{\nu_1}\cdot\nabla\Scal_{\widetilde{\Gamma}_{N}} \\
-\widetilde{\nu_2}\cdot\nabla\Scal_{\widetilde{\Gamma}_1} &\lambda_2-\Kcal_{\widetilde{\Gamma}_2}^*  & \cdots & -\widetilde{\nu_2}\cdot\nabla\Scal_{\widetilde{\Gamma}_{N}}\\
\vdots & \vdots &\ddots &\vdots \\
-\widetilde{\nu_{N}}\cdot\nabla\Scal_{\widetilde{\Gamma}_1} & -\widetilde{\nu_{N}}\cdot\nabla\Scal_{\widetilde{\Gamma}_2} & \cdots & \lambda_{N}-\Kcal_{\widetilde{\Gamma}_{N}}^*
\end{bmatrix}.
\end{split}
\end{equation}
From the identities \eqref{npae} and \eqref{dsae}, the linearity of the equation \eqref{eq:integralrep} and together with the help of the following relationship
$$
\frac{\partial}{\partial \nu}H(y)=\frac{\partial}{\partial \widetilde{\nu}}\sum^{+\infty}_{|\beta|=1} \frac{1}{\beta!}  \widetilde{y}^{\beta} \partial^{\beta}H(z),
$$
one can conclude that $\widetilde{\phi}_{k}$, $k=1,2,\ldots,N$, with the following expression
\begin{equation}\label{eq:phipsi01}
\widetilde{\phi}_k=\sum_{|\beta|=1}^{+\infty} \frac{1}{\beta!} \widetilde{\phi}_{k,\beta}\partial^{\beta}H(z),
\end{equation}
is the solution of \eqref{eq:integralrep}. Therefore from \eqref{eq:general_solution}, we have the following expansion for the  perturbed electric potential $u-H$,
\begin{equation}\label{eq:farexp01}
u(x)-H(x) = \sum_{k=1}^{N}\sum^{+\infty}_{|\alpha|=1} \sum^{+\infty}_{|\beta|=1}  \frac{(-1)^{|\alpha|}}{\alpha!\beta!}\partial^{\alpha} G(x-z)\partial^{\beta} H(z)\int_{\widetilde{\Gamma}_k} \widetilde{y}^{\alpha} \widetilde{\phi}_{k,\beta}(\widetilde{y})~\mathrm{d}s(\widetilde{y}).
\end{equation}
Then  we can obtain the following result.

\begin{thm}\label{th:farfexp01}
Let $u(x)$ be the solution to the problem \eqref{eq:mainmd02} with the conductivity $\sigma $ given by \eqref{eq:paracho01} and the transmission conditions given by \eqref{eq:transmscds}. Then there holds
\begin{equation}\label{eq:lemag0101}
u(x)-H(x) = \sum^{+\infty}_{|\alpha|=1} \sum^{+\infty}_{|\beta|=1}  \frac{(-1)^{|\alpha|}}{\alpha!\beta!}\partial^{\alpha} G(x-z)\widetilde M_{\alpha\beta}\partial^{\beta} H(z),
\end{equation}
where $\widetilde M_{\alpha\beta}$ is defined  in \eqref{eq:def_M2} with the integral surfaces replaced by $\widetilde{\Gamma}_k$. Correspondingly, the first-order polarization tensor is  $\widetilde{\mathbf{M}} =(\widetilde M_{ij})_{i,j=1}^{d}.$
\end{thm}

\subsection{Uniqueness of the location for multi-layer structures}
We are in a position to present the unique recovery results in
locating the multi-layer structure. In what follows, we let $A^{(1)}=\cup_{k=1}^NA_k^{(1)}$ and $A^{(2)}=\cup_{k=1}^NA_k^{(2)}$,  be two $N$-layer structure, which satisfy \eqref{eq:form1} with $z$ replaced by $z^{(1)}$ and $z^{(2)}$, respectively. Correspondingly, the material parameter
$
\sigma_k
$, $k=1,2,\ldots,N$,
is replaced by
$
\sigma^{(1)}_k
$
and
$
\sigma^{(2)}_k,
$
respectively, for $A^{(1)}$ and $A^{(2)}$. Let $u_j$, $j=1, 2$, be the solutions to \eqref{eq:mainmd02} with $A$ replaced by $A^{(1)}$ and $A^{(2)}$, respectively. Denote by $\widetilde{\mathbf{M}}_{1}$, $\widetilde{\mathbf{M}}_{2}$ the  first-order polarization tensors for $A^{(1)}$ and $A^{(2)}$, respectively.
\begin{thm}\label{th:02}
	Let $\Omega$ be a bounded domain enclosing $A^{(1)}\cup A^{(2)}$.  Suppose that $\nabla H(x)\neq 0$ for all $x\in \Omega$, and either  $\widetilde{\mathbf{M}}_{1}$ or $\widetilde{\mathbf{M}}_{2}$ is  nonsingular. If
	\begin{equation}\label{measure}
	u_1=u_2 \ \ \mbox{on} \ \ \Pi,
	\end{equation}	
	then
	\[
	z^{(1)}=z^{(2)},
	\] where $\Pi$ is an open subset of $\p \Omega$.
\end{thm}
\begin{proof}[\bf Proof]
	Since $u_1$ and $u_2$ are harmonic in $\RR^d\setminus\overline{\Omega}$ ($d=2,3$), by using \eqref{measure} and unique continuation, one has that
	\[
	u_1=u_2 \ \ \mbox{in} \ \ \RR^d\setminus\overline{\Omega}.
	\]
	Then from Theorem \ref{th:farfexp01}, there holds that, for $x\in\RR^d\setminus \overline{\Omega}$,
	\[
	u_j(x)=
	H(x)-\nabla G(x-z^{(j)})^T  \widetilde{\mathbf{M}}_{j}\nabla H(z^{(j)})+\mathcal{O}\(\frac{1}{|x-z^{(j)}|^d}\), \quad j=1,2,
	\]
	which implies that
	\[
	  \nabla G(x-z^{(1)})^T  \widetilde{\mathbf{M}}_{1}\nabla H(z^{(1)})
	-
	\nabla G(x-z^{(2)})^T  \widetilde{\mathbf{M}}_{2}\nabla H(z^{(2)}) =0 \ \ \mbox{in} \ \ \RR^d\setminus\overline{\Omega}.
	\]
	By straightforward calculations, one can further show that
	\begin{equation}\label{eq:reviseadd01}
	\begin{split}
	F(x):=&\(\nabla G(x-z^{(1)})-\nabla G(x-z^{(2)})\)^T\widetilde{\mathbf{M}}_{1}\nabla H(z^{(1)})\\
	&-\nabla G(x-z^{(2)})^T\(\widetilde{\mathbf{M}}_{2}\nabla H(z^{(2)})-\widetilde{\mathbf{M}}_{1}\nabla H(z^{(1)}\)\\
	=&\(\nabla^2 G(x-z')(z^{(1)}-z^{(2)})\)^T\widetilde{\mathbf{M}}_{1}\nabla H(z^{(1)})\\
	&-\nabla G(x-z^{(2)})^T\(\widetilde{\mathbf{M}}_{2}\nabla H(z^{(2)})-\widetilde{\mathbf{M}}_{1}\nabla H(z^{(1)}\)=0
	\end{split}
	\end{equation}
	holds in $\RR^d\setminus\overline{\Omega}$, where $z'=z^{(1)}+t'z^{(2)}$ with $t'\in (0, 1)$. Note that $F(x)$ defined in \eqref{eq:reviseadd01} is also harmonic in $\RR^d\setminus(z^{(1)}\cup z^{(2)})$. By using the analytic continuation of harmonic functions, one thus has that $F(x)\equiv 0$ in $\RR^d$. Define $F:=F_1+F_2$, where
	\begin{equation}\label{equf1}
		F_1(x):=\(\nabla^2 G(x-z')(z^{(1)}-z^{(2)})\)^T\widetilde{\mathbf{M}}_{1}\nabla H(z^{(1)}),
	\end{equation}
	and
	\[
	F_2(x):=-\nabla G(x-z^{(2)})^T\(\widetilde{\mathbf{M}}_{2}\nabla H(z^{(2)})-\widetilde{\mathbf{M}}_{1}\nabla H(z^{(1)}\).
	\]
	Then by comparing the types of poles of $F_1$ and $F_2$, one immediately finds that $F_1=0$ and $F_2=0$ in $\RR^d$.
	If $\nabla H(x)\neq 0$ for all $x\in \Omega$ and  $\widetilde{\mathbf{M}}_{1}$  is  nonsingular, it follows from $F_1=0$ that 	
	\[
	z^{(1)}-z^{(2)}=0.
	\]
	On the other hand, 	similarly to \eqref{eq:reviseadd01}, we can obtain
	\[
	\(\nabla^2 G(x-z')(z^{(1)}-z^{(2)})\)^T\widetilde{\mathbf{M}}_{2}\nabla H(z^{(2)})=0.
	\]
	If $\nabla H(x)\neq 0$ for all $x\in \Omega$ and  $\widetilde{\mathbf{M}}_{2}$  is  nonsingular, it also follows that 	
	$
	z^{(1)}-z^{(2)}=0.
	$
	The proof is complete.
\end{proof}
\begin{rem}
We remark that the assumption $\nabla H(x)\neq 0$ for all $x\in \Omega$, which has  been used as a crucial condition in \cite[Page 183]{HK07:book} to establish the results therein, is also a requirement in Theorem \ref{th:02}. Moreover, the first-order polarization tensor may be singular, if we only have the elliptic assumption on the conductivities $\sigma_k$, $k = 1,2,\ldots,N.$  In this regards, we refer to \cite{KLSIAM2019,KLSAA2020} for construction of the polarization tensor vanishing structure to achieve weakly neutral inclusions.
We would like to emphasize that the uniqueness result of Theorem \ref{th:02} also holds if we assume that $\Gs_k-1>0$ or $\Gs_k-1<0$ for all $k =1,2,\ldots,N$, and $\nabla H(x)\neq 0$ for all $x\in \Omega$.
 This, together with the fact that $\widetilde{\mathbf{M}}_{1}$ is a nonsingular (actually positive- or negative--definite) matrix (see, Theorem \ref{Positivedef}),  implies that
$z^{(1)}-z^{(2)}=0.$
\end{rem}

\section{Reconstruction of the conductivity distribution for multi-layer concentric disks}\label{sec4}
For multi-layer structure, we are mainly concerned with the following inverse conductivity problem:
\[
\left.(u, H)\right|_{x\in \Pi}\longrightarrow \bigcup_{k=1}^N\(A_k;\sigma_k,\Gamma_k\),
\]
where $\Pi$ is an open surface outside the multi-layer structure. We shall only consider the two dimensional case.

In what follows, for later usage, we introduce some notions on the measurements.
\begin{defn}
	\label{th:expan2} Let $H$ be a harmonic function in $\RR^2$, which admits the following expansion
	\begin{equation}\label{Hexp}
	H(x) = H(0)+\sum_{n=1}^\infty r^n \bigr(a_n^c\cos n\theta + a_n^s\sin n\theta \bigr).
	\end{equation}
	We call $H$ is of \textit{full-order}, if the expansion (\ref{Hexp}) hold such that
	\[
		a_n^c\neq 0, \quad a_n^s\neq 0, \mbox{ for all } n\in \mathbb{N}.
	\]
	Otherwise $H$ is of \textit{partial-order}. Furthermore, in (\ref{eq:mainmd02}), if $H$ is of \textit{full-order}, then we call
	the inverse conductivity problem has \textit{full-order} measurement. Otherwise it has  \textit{partial-order} measurement.
\end{defn}

We mention that lots of harmonic functions can be of $\textit{full-order}$. For example, consider a complex valued function
$f(z)=e^z$, where $z=x+\rm{i}y$ with $\rm{i}$ the imaginary unit, that is $\rm{i}^2=-1$. It is readily seen, by Taylor expansion, that any nontrivial combination of
real part and imaginary part of $f(z)$ is of \textit{full-order} measurement.

In order to reconstruct the conductivity distribution for multi-layer structure  by using \emph{partial-order} measurement,
we next seek an expression of the multipolar expansion in $\mathbb{R}^2$ which is slightly different from \eqref{FFEPE}.  For multi-indices $\alpha\in \mathbb{N}^2$,
define $a_{\alpha}^c$ and $a_{\alpha}^s$ by
\[
\sum_{|\alpha|=n}a_{\alpha}^c x^{\alpha} = r^n\cos n\theta \quad \mbox{ and }\quad\sum_{|\alpha|=n}a_{\alpha}^s x^{\alpha} = r^n\sin n\theta,
\]
and define the contracted GPTs of multi-layer structures
\begin{align}
\ds  M_{mn}^{cc}:=\sum_{|\alpha|=m} \sum_{|\beta|=n}a_{\alpha}^c a_{\beta}^c M_{\alpha\beta}, \label{defmcc1}\\
\ds  M_{mn}^{cs}:=\sum_{|\alpha|=m} \sum_{|\beta|=n}a_{\alpha}^c a_{\beta}^s M_{\alpha\beta}, \\
\ds  M_{mn}^{sc}:=\sum_{|\alpha|=m} \sum_{|\beta|=n}a_{\alpha}^s a_{\beta}^c M_{\alpha\beta}, \\
\ds  M_{mn}^{ss}:=\sum_{|\alpha|=m} \sum_{|\beta|=n}a_{\alpha}^s a_{\beta}^s M_{\alpha\beta}. \label{defmss1}
\end{align}
Note that $G(x-y)$ admits the expansion
\begin{equation}\label{Gammaexp}
G(x-y)
=\sum_{n=1}^{\infty}\frac{-1}{2\pi n}\left[\frac{\cos n\theta_x}{r_x^n}r_y^n\cos n\theta_y
+\frac{\sin n\theta_x}{r_x^n}r_y^n\sin n\theta_y\right]+ C,
\end{equation}
where $C$ is a constant, $x=r_x(\cos\theta_x,\sin\theta_x)$ and
$y=r_y(\cos\theta_y,\sin\theta_y)$. Expansion \eqref{Gammaexp} is
valid if $|x| \to +\infty$ and $y \in \Gamma_k$.

From \eqref{rep_solu} and \eqref{Gammaexp}, we get the
following theorem.
\begin{thm}\label{th:expan3}
	Let $u$ be the solution to \eqref{eq:mainmd02}  in $\mathbb{R}^2$ with the conductivity $\sigma $ given by \eqref{eq:paracho01} and the transmission conditions given by \eqref{eq:transmscds}. If $H$ admits the expansion
	\begin{equation}\label{Hexp1}
	H(x) = H(0)+\sum_{n=1}^\infty r^n \bigr(a_n^c\cos n\theta + a_n^s\sin n\theta \bigr)
	\end{equation}
	with $x=(r\cos\theta, r\sin\theta)$,
	then we have
	\begin{align}
	(u-H)(x) & = -\sum_{m=1}^\infty\frac{\cos m\theta}{2\pi mr^m}\sum_{n=1}^\infty
	\( M_{mn}^{cc}a_n^c + M_{mn}^{cs}a_n^s \) \nonumber \\
	& \quad -\sum_{m=1}^\infty\frac{\sin m\theta}{2\pi m
		r^m}\sum_{n=1}^\infty \bigr( M_{mn}^{sc}a_n^c + M_{mn}^{ss}a_n^s
	\bigr), \label{expan3}
	\end{align}
	which holds uniformly as $|x| \to +\infty$.
\end{thm}

The  CGPTs \eqref{defmcc1}--\eqref{defmss1} involving geometric and material configurations of multi-layer structure play an important role in reconstructing conductivity distributions. Unfortunately for general shape they are coupled together and difficult to decouple.
It is proved in \cite{ADKL14} that the full set of harmonic combinations  of CGPTs associated with an inhomogeneous inclusion determines the Newmann-to-Dirichlet map on the boundary of the inclusion. Then  uniqueness results of the Calder\'on  problems hold for conductivities in $L^{\infty}$ (see \cite{APam2006}).
Motivated by the above facts and results, in the remainder of this section, we shall consider the uniqueness recovery of conductivity distribution for multi-layer concentric disks  by using \emph{partial-order} measurement. We also want to remark that in the multi-layer concentric disks case, by using \emph{full-order} measurement, one can readily recover all the CGPTs and thus the multi-layer structure.

We suppose that $A$ is a multi-layer concentric disks in $\RR^2$. Precisely, we give a sequence of layers, $A_0,A_1,\ldots,A_{N}$, by
\begin{equation}\label{eq:aj}
A_{0}:=\{r>r_{1}\}, \quad A_k:=\{r_{k+1}<r\leqslant r_{k}\}, \quad  k=1,2,\ldots, N-1, \quad A_N:=\{r\leqslant r_N\},
\end{equation}
and the interfaces between the adjacent layers can be rewrite by
\begin{equation}\label{interface}
\Gamma_k:=\left\{|x|=r_k\right\}, \quad  k=1,2,\ldots,N,
\end{equation}
where $N\in \mathbb{N}$ and $r_k\in\mathbb{R}_+$.

\subsection{Explicit formulae for the polarization tensors of multi-layer concentric disks}

In this subsection, we explicitly compute the solution $\phi_k$ of the integral equation \eqref{solu2} in the case where the inclusion $A$ is $N$-layer concentric disk.

 Let $\Gamma_0=\{ |x|=r_0\}$. For each integer $n$, one can easily see that (cf. \cite{ACKLM1})
\begin{equation} \label{Single_circular}
\Scal_{\Gamma_0}[e^{\mathrm{i} n\theta}](x) = \begin{cases}
\ds - \frac{r_0}{2|n|} \left(\frac{r}{r_0}\right)^{|n|}
e^{\mathrm{i} n\theta} \quad & \mbox{if } |x|=r < r_0, \\
\nm
\ds - \frac{r_0}{2|n|} \left(\frac{r_0}{r}\right)^{|n|}
e^{\mathrm{i} n\theta} \quad & \mbox{if } |x|=r > r_0,
\end{cases}
\end{equation}
and hence
\begin{equation}\label{single-cir-nor}
\frac{\p }{\p r} \Scal_{\Gamma_0}[e^{\mathrm{i} n\theta}](x) =  \begin{cases}
\ds - \frac{1}{2} \left(\frac{r}{r_0}\right)^{|n|-1}
e^{\mathrm{i} n\theta} \quad & \mbox{if } |x|=r < r_0, \\
\nm
\ds \frac{1}{2} \left(\frac{r_0}{r}\right)^{|n|+1}
e^{\mathrm{i} n\theta} \quad & \mbox{if } |x|=r > r_0.
\end{cases}
\end{equation}
It then follows from \eqref{eq:trace} that
\begin{equation}\label{kcal-circ}
\Kcal_{\Gamma_0}^* [e^{\mathrm{i} n\theta}] =0 \quad \forall n \neq 0.
\end{equation}

Our main result in this subsection is the following.

\begin{thm}\label{thm52}
Let the multi-layer concentric disks $A=\cup_{k=1}^NA_k$ be given by \eqref{eq:aj}--\eqref{interface}. Assume that the background electrical potential $H$ can be represented as
\begin{equation}\label{eq:defH01}
H= \sum_{n=1}^{+\infty} a_{n}r^n e^{\mathrm{i} n\theta},
\end{equation}
where $a_n,$ $(n = 1, 2,\ldots),$ are arbitrary constants.
Then, the solution of \eqref{solu2} is given by
\begin{equation}\label{PHINK1}
\phi_k=2\sum_{n=1}^{+\infty}na_{n}e^{\mathrm{i} n\theta}\bm{e}_k^T\(\mathbb{E}^{(n)}_N\)^{-1}\bm{e},
\end{equation}
where
\[
\begin{split}
\mathbb{E}_N^{(n)}:=
\begin{bmatrix}
2\lambda_{1}r_1^{1-n} & -  r_2^{n+1}r_1^{-2n} & -r_3^{n+1}r_1^{-2n}& \cdots & - r_{N-1}^{n+1}r_1^{-2n}& -r_N^{n+1}r_1^{-2n} \\
r_1^{1-n} & 2\lambda_{2}r_2^{1-n}&- r_3^{n+1}r_2^{-2n}& \cdots & - r_{N-1}^{n+1}r_2^{-2n}& -r_N^{n+1}r_2^{-2n}\\
r_1^{1-n} &   r_2^{1-n}& 2\lambda_{3}r_3^{1-n} & \cdots & - r_{N-1}^{n+1}r_3^{-2n}& -r_N^{n+1}r_3^{-2n}\\
\vdots & \vdots &\vdots &\ddots& \vdots &\vdots \\
r_1^{1-n} &  r_2^{1-n}&r_3^{1-n} & \cdots &2\lambda_{N-1}r_{N-1}^{1-n}&-r_N^{n+1}r_{N-1}^{-2n}\\
r_1^{1-n} &  r_2^{1-n}&r_3^{1-n}& \cdots & r_{N-1}^{1-n}& 2\lambda_{N}r_N^{1-n}
\end{bmatrix}.
\end{split}
\]
\end{thm}
\begin{proof}[\bf Proof]
Because of \eqref{kcal-circ} it follows that
\[
\begin{split}
\mathbb{K}_A^*:=
\begin{bmatrix}
0 & \nu_1\cdot\nabla\Scal_{\Gamma_2} & \cdots & \nu_1\cdot\nabla\Scal_{\Gamma_{N}} \\
\nu_2\cdot\nabla\Scal_{\Gamma_1} & 0 & \cdots & \nu_2\cdot\nabla\Scal_{\Gamma_{N}}\\
\vdots & \vdots &\ddots &\vdots \\
\nu_{N}\cdot\nabla\Scal_{\Gamma_1} & \nu_{N}\cdot\nabla\Scal_{\Gamma_2} & \cdots & 0
\end{bmatrix}.
\end{split}
\]
From \eqref{single-cir-nor}, if $\bm\phi$ is given by
\begin{equation}\label{PHI}
\bm\phi = \sum_{n=1}^{+\infty} (\phi^n_1e^{\mathrm{i} n\theta},\phi^n_2e^{\mathrm{i} n\theta},\ldots,\phi^n_Ne^{\mathrm{i} n\theta})^T,
\end{equation}
then the integral equations \eqref{eq:integralrep} are equivalent to
\[
\left\{
\begin{aligned}
& \lambda_{1} \phi^n_1-\frac{1}{2}\sum_{k=2}^{N}\phi^n_k\(\frac{r_k}{r_1}\)^{n+1}=na_{n}r_1^{n-1},\\
&\frac{1}{2}\sum_{k=1}^{l-1}\phi^n_k\(\frac{r_l}{r_k}\)^{n-1}+\lambda_{l}\phi^n_l-\frac{1}{2}\sum_{k=l+1}^{N}\phi^n_k\(\frac{r_k}{r_l}\)^{n+1}=na_{n}r_l^{n-1}, l=2,3,\ldots,N-1,\\
&\frac{1}{2}\sum_{k=1}^{N-1}\phi^n_k\(\frac{r_N}{r_k}\)^{n-1}+\lambda_{N}\phi^n_N=na_{n}r_N^{n-1}.
\end{aligned}
\right.
\]
It follows that
\[
\left\{
\begin{aligned}
&2 \lambda_{1} \phi^n_1 r_1^{1-n}-{r_1}^{-2n}\sum_{k=2}^{N}\phi^n_k{r_k}^{n+1}=2na_{n},\\
&\sum_{k=1}^{l-1}\phi^n_k r_k^{1-n} +2\lambda_{l}\phi^n_l r_l^{1-n}-{r_l}^{-2n}\sum_{k=l+1}^{N}\phi^n_k{r_k}^{n+1}=2na_{n}, l=2,3,\ldots,N-1,\\
&\sum_{k=1}^{N-1}\phi^n_k r_k^{1-n} +2\lambda_{N}\phi^n_N r_N^{1-n}=2na_{n}.
\end{aligned}
\right.
\]
Therefore,
we can obtain that
\[
\mathbb{E}_N^{(n)}\((\phi^n_1,\phi^n_2,\ldots,\phi^n_N)^T\)=2na_{n}\bm{e},
\]
where $\bm{e}:=(1,1,\ldots,1)^T$. It is clear that the invertibility of the  matrix $\mathbb{E}_N^{(n)}$ is equivalent to the well-posedness of the conductivity problem \eqref{eq:mainmd02} with all the material parameters $\sigma_k$, $k=1,2,\ldots, N$ being positive.
Thus, we can deduce that
\begin{equation}\label{PHINK}
\phi_k=2\sum_{n=1}^{+\infty}na_{n}e^{\mathrm{i} n\theta}\bm{e}_k^T\(\mathbb{E}^{(n)}_N\)^{-1}\bm{e}.
\end{equation}
The proof is complete.
\end{proof}

As an immediate application of the above theorem we obtain the following explicit
form of the perturbed electric potential in entire two-dimensional space $\RR^2$ in terms of the \emph{generalized polarization matrix}.
\begin{thm}\label{th:solmain01}
	Let $A=\cup_{k=1}^{N}A_k$ be the multi-layer concentric disk given by \eqref{eq:aj}.	Suppose $u$ is the solution to \eqref{eq:mainmd02} with the conductivity $\sigma $ given by \eqref{eq:paracho01} and the transmission conditions given by \eqref{eq:transmscds}. Let $H$ be given by \eqref{eq:defH01}. Define the   $n$-order generalized polarization matrix (GPM) $\mathbb{M}_N^{(n)}$ as follows:
	\begin{equation}\label{eq:matP01}
	\mathbb{M}_N^{(n)}:= \begin{bmatrix}
	-2\Gl_1 & (r_{2}/r_1)^{2n} & (r_{3}/r_1)^{2n} & \cdots& (r_{N-1}/r_1)^{2n} &(r_{N}/r_1)^{2n} \\
	-1 & -2\Gl_{2} & (r_{3}/r_2)^{2n} &\cdots& (r_{N-1}/r_2)^{2n}  & (r_{N}/r_2)^{2n} \\
	-1 & -1 & -2\Gl_{3} &\cdots& (r_{N-1}/r_3)^{2n}  & (r_{N}/r_3)^{2n} \\
	\vdots & \vdots & \vdots & \ddots & \vdots& \vdots\\
	-1 & -1 & -1 & \cdots & -2\Gl_{N-1}&(r_{N}/r_{N-1})^{2n} \\
	-1 & -1& -1 &\cdots & -1&-2\Gl_{N}
	\end{bmatrix}.
	\end{equation}
	Then $\mathbb{M}_N^{(n)}$ is invertible, and the transmission problem \eqref{eq:mainmd02} is uniquely solvable with the solution given by the following formula:
	\begin{equation}\label{eq:purbmn01}
	u-H=\sum_{n=1}^{+\infty}a_{n}e^{\mathrm{i} n\theta}\(r^{n}\bm{e}_{1:l}^T
	+\frac{1}{r^{n}}\bm{e}_{l+1:N}^T
	\)\Upsilon_{N,l}^{(n)}  (\mathbb M_{N}^{(n)})^{-1}\bm{e} \quad \mbox{in } A_l,\; l =0,1,\ldots,N,
	\end{equation}
	where $\bm{e}_{i:j} = \sum_{k=i}^{j}\bm{e}_k$ for $i\leqslant j$, and set $\bm{e}_{i:j} =0$ for $i>j$, and
	\begin{equation}\label{eq:matU01}
	\Upsilon_{N,l}^{(n)}:= \begin{bmatrix}
	1&\cdots & 0 & 0 &\cdots& 0 \\
	\vdots&\ddots & \vdots & \vdots &\ddots & \vdots\\
	0&\cdots & 1 & 0 &\cdots& 0 \\
	0&\cdots & 0 & r_{l+1}^{2n} &\cdots& 0\\
	\vdots & \ddots & \vdots & \vdots& \ddots & \vdots\\
	0&\cdots & 0 &
	0&\cdots& r_{N}^{2n}\\
	\end{bmatrix}.
	\end{equation}
\end{thm}
\begin{proof}[\bf Proof]
	It then follows from Theorem \ref{thm52}, \eqref{Single_circular} and \eqref{PHINK} that
	the perturbed electric potential $u-H$ 	in $A_l$, $l =0,1,\ldots,N$, can be given by
	\[
	\begin{aligned}
	u-H&=2\sum_{n=1}^{+\infty}na_{n}\sum_{k=1}^{N}\Scal_{\Gamma_k}[e^{\mathrm{i} n\theta}]\bm{e}_k^T\(\mathbb{E}^{(n)}_N\)^{-1}\bm{e}\\
	&=-\sum_{n=1}^{+\infty}a_{n}\(r^{n}e^{\mathrm{i} n\theta}\sum_{k=1}^{l} \frac{1}{r_k^{n-1}}
	+\frac{e^{\mathrm{i} n\theta}}{r^{n}}\sum_{k=l+1}^{N} r_k^{n+1}
	\)\bm{e}_k^T\(\mathbb{E}^{(n)}_N\)^{-1}\bm{e}\\
	&=-\sum_{n=1}^{+\infty}a_{n}e^{\mathrm{i} n\theta}\(r^{n}\bm{e}_{1:l}^T
	+\frac{1}{r^{n}}\bm{e}_{l+1:N}^T
	\)\mathbb{F}^{(n)}_{N}\(\mathbb{E}^{(n)}_N\)^{-1}\bm{e}\\
	&=\sum_{n=1}^{+\infty}a_{n}e^{\mathrm{i} n\theta}\(r^{n}\bm{e}_{1:l}^T
	+\frac{1}{r^{n}}\bm{e}_{l+1:N}^T
	\)\Upsilon_{N,l}^{(n)}  (\mathbb M_{N}^{(n)})^{-1}\bm{e},
	\end{aligned}
	\]
	where
	\begin{equation}
	\mathbb{F}^{(n)}_{N,l}:= \begin{bmatrix}
	r_1^{-n+1}&\cdots & 0 & 0 &\cdots& 0 \\
	\vdots&\ddots & \vdots & \vdots &\ddots & \vdots\\
	0&\cdots & r_l^{-n+1} & 0 &\cdots& 0 \\
	0&\cdots & 0 & r_{l+1}^{n+1} &\cdots& 0\\
	\vdots & \ddots & \vdots & \vdots& \ddots & \vdots\\
	0&\cdots & 0 &
	0&\cdots& r_{N}^{n+1}\\
	\end{bmatrix}.
	\end{equation}
	The proof is complete.
\end{proof}

\begin{rem}
When  $a_1\neq 0$, $a_n=0$ for $n=2,3,\ldots,$ and  $N=2$, the condition $u-H = 0$ in $A_0$ in \eqref{eq:purbmn01} leading to the cloaking of a two-layed concentric disk gives the Hashin-Shtrikman formula \cite{GWM2022}
\begin{equation}\label{HSF}
\sigma_0 = \sigma_1 + \frac{2\sigma_1f_1(\sigma_2-\sigma_1)}{2\sigma_1+f_2(\sigma_2-\sigma_1)},
\end{equation}
where $f_1 =1-f_2 = \frac{r_2^2}{r_1^2}$. This suggests that there may be an  effective conductivity $\sigma_0$ at which the current is neither attracted nor diverted around the inclusion but remains completely unperturbed in the exterior region, which is equivalent to the first order polarization tensors of the inclusion vanishing. In other words, inserting this two-layed concentric disk into the matrix would not disturb the uniform current outside the disk, and Hashin's neutral inclusion is a GPT-vanishing structure of order 1. The formula \eqref{eq:purbmn01} might provide  a new  perspective on the design of GPT-vanishing structures of $N-1$ order by using $N$-layer concentric disks.
\end{rem}

\subsection{Uniqueness of the conductivity distribution of multi-layer concentric disks}
We shall consider the unique recovery of the conductivity distribution, i.e., the surfaces $r_k$ and the medium parameters $\sigma_k$, $k=1,2,\ldots,N$. To this end, let $A^{(j)}=\cup_{k=1}^NA_k^{(j)}$, $j=1,2$,  be two $N$-layer concentric disks, which satisfy \eqref{eq:aj} with $r_k$ replaced by $r_k^{(1)}$ and $r_k^{(2)}$, respectively. Correspondingly, the material parameter
$
\sigma_k
$, $k=1,2,\ldots,N$,
is replaced by
$
\sigma^{(1)}_k
$
and
$
\sigma^{(2)}_k,
$
respectively, for $A^{(1)}$ and $A^{(2)}$. Let $u_j$, $j=1, 2$, be the solutions to \eqref{eq:mainmd02} with $A$ replaced by $A^{(1)}$ and $A^{(2)}$, respectively. Denote by $\mathbb{M}_{N,1}^{(n)}$, $\mathbb{M}_{N,2}^{(n)}$ the  $n$-order \emph{GPM} for $A^{(1)}$ and $A^{(2)}$, respectively.

From \eqref{eq:purbmn01}, there holds the following for $x\in A_0$,
\[
u_j=H+\bm{e}^T\sum_{n=1}^{+\infty}a_{n}\frac{e^{\mathrm{i} n\theta}}{r^{n}}\Upsilon_{N,j}^{(n)} (\mathbb M_{N,j}^{(n)})^{-1}\bm{e}, \quad j=1,2.
\]
In order to obtain the uniqueness recovery of conductivity distribution, we shall study the row vector $\bm{e}^T \Upsilon_{N}^{(n)} (\mathbb M_{N}^{(n)})^{*} $ and the column vector $(\mathbb M_{N}^{(n)})^{*}\bm{e}$, where superscript $*$ denotes the adjugate of a matrix. In order to simplify the analysis, in our subsequent study, we
always assume that
$t^n_{i,j}=(r_{j}/r_{i})^{2n}$
\[
K_M^{i,j}(n):=
\begin{vmatrix}
t_{i,j}^n+1 & t_{i,j+1}^n+1 & \cdots& t_{i,M-1}^n+1 &t_{i,M}^n+2\Gl_{M} \\
-2\Gl_{j}+1 & t_{j,j+1}^n+1 &\cdots& t_{j,M-1}^n+1  & t_{j,M}^n+2\Gl_{M} \\
0 & -2\Gl_{j+1}+1 &\cdots& t_{j+1,M-1}^n+1  & t_{j+1,M}^n+2\Gl_{M} \\
\vdots & \vdots& \ddots & \vdots  & \vdots\\
0 & 0 & \cdots & -2\Gl_{M-1}+1&t_{M-1,M}^n+2\Gl_{M}
\end{vmatrix},
\]
and
\[
\begin{aligned}
L_{M}^{i,j}(n)&:= \begin{vmatrix}
r_i^{2n} & r_{j}^{2n} & r_{j+1}^{2n} & \cdots& r_{M-1}^{2n} &r_{M}^{2n} \\
-1-t^n_{j,i} & -2\Gl_{j}-1 & 0 &\cdots& 0  & 0 \\
-1-t^n_{j+1,i} & -1-t^n_{j+1,j} & -2\Gl_{j+1}-1 &\cdots& 0  & 0 \\
\vdots & \vdots & \vdots & \ddots & \vdots& \vdots\\
-1-t^n_{M-1,i} & -1-t^n_{M-1,j} & -1-t^n_{M-1,j+1} & \cdots & -2\Gl_{M-1}-1&0 \\
-1 & -1& -1 &\cdots & -1&-2\Gl_{M}
\end{vmatrix}\\
\end{aligned}
\]
where $i<j$ and set
\[
K_M^{i,M+1}=1,\;\mbox{ and }\;L_M^{i,M+1}=r_i^{2n}.
\]
By direct computations, one can derive the  recursion formulae for $K_M^{i,j}$ and $L_M^{i,j}$ in the following lemma, respectively.
\begin{lem}\label{recukl}
	There holds the following recursion formulae:
	\begin{equation}\label{recukij}
	K_M^{i,j}=\( t_{i,j}^n+1\)K_M^{j,j+1}-\(-2\lambda_{j}+1\)K_M^{i,j+1},
	\end{equation}
	and
	\begin{equation}\label{reculij}
	L_M^{i,j}=\(t^n_{j,i}+1\)L_M^{j,j+1}+\(-2\lambda_{j}-1\)L_M^{i,j+1}.
	\end{equation}
\end{lem}

Next, we give  the explicit formulae for  each element of the row vector $\bm{e}^T \Upsilon_{N}^{(n)} (\mathbb M_{N}^{(n)})^{*} $ and the column vector $(\mathbb M_{N}^{(n)})^{*}\bm{e}$.
\begin{lem}\label{gtf}
	The general term formulae for each element of the row vector $\bm{e}^T \Upsilon_{N}^{(n)} (\mathbb M_{N}^{(n)})^{*} $ and the column vector $(\mathbb M_{N}^{(n)})^{*}\bm{e}$ can be represented by
	\begin{equation}\label{pei}
	\((\mathbb M_{N}^{(n)})^{*}\bm{e}\)_i
	=(-1)^{N-i}\prod_{j=1}^{i-1}\(-2\lambda_j+1\)K_N^{i,i+1}(n),
	\end{equation}
	and
	\begin{equation}\label{erpi}
	\(\bm{e}^T \Upsilon_{N}^{(n)} (\mathbb M_{N}^{(n)})^{*}\)_i
	=\prod_{j=1}^{i-1}\(-2\lambda_j-1\)L_N^{i,i+1}(n),
	\end{equation}
	respectively, where $i=1,2,\ldots,N$.
\end{lem}
\begin{proof}[\bf Proof]	
	By using the Laplace expansion theorem for determinant, one can derive that  $\((\mathbb M_{N}^{(n)})^{*}\bm{e}\)_i$  is equal to the determinant after replacing the $i$-th column of the matrix $\mathbb M_{N}^{(n)}$  with the vector $\bm{e}$.
	With the help of this fact and some elementary transformation, we can obtain
	\[
	\begin{aligned}
	\((\mathbb M_{N}^{(n)})^{*}\bm{e}\)_{i}&= \begin{vmatrix}
	-2\Gl_{1} & t_{1,2}^n  & \cdots& t_{1,i-1}^n & 1 & t_{1,i+1}^n& \cdots & t_{1,N-1}^n &t_{1,N}^n \\
	-1 & -2\Gl_{2}  &\cdots& t_{2,i-1}^n& 1& t_{2,i+1}^n& \cdots& t_{2,N-1}^n  & t_{2,N}^n \\
	\vdots & \vdots  & \ddots & \vdots& \vdots& \vdots&\ddots& \vdots& \vdots\\
	-1 & -1  & \cdots & -2\lambda_{i-1}& 1& t_{i-1,i+1}^n&\cdots& t_{i-1,N-1}^n& t_{i-1,N}^n\\
	-1 & -1  & \cdots & -1& 1& t_{i,i+1}^n&\cdots& t_{i,N-1}^n& t_{i,N}^n\\
	-1 & -1  & \cdots & -1& 1& -2\lambda_{i+1}&\cdots& t_{i+1,N-1}^n& t_{i+1,N}^n\\
	\vdots & \vdots  & \ddots & \vdots& \vdots& \vdots&\ddots& \vdots& \vdots\\
	-1 & -1  & \cdots& -1 & 1& -1& \cdots& -2\lambda_{N-1}&t_{N-1,N}^n\\
	-1 & -1 &\cdots& -1 & 1& -1& \cdots& -1&-2\lambda_N
	\end{vmatrix}\\
	&= \begin{vmatrix}
	-2\Gl_{1}+1 & t_{1,2}^n+1  & \cdots& t_{1,i-1}^n+1 & 1 & t_{1,i+1}^n+1& \cdots & t_{1,N-1}^n +1&t_{1,N}^n \\
	0 & -2\Gl_{2}+1  &\cdots& t_{2,i-1}^n+1& 1& t_{2,i+1}^n+1& \cdots& t_{2,N-1}^n+1  & t_{2,N}^n \\
	\vdots & \vdots  & \ddots & \vdots& \vdots& \vdots&\ddots& \vdots& \vdots\\
	0 & 0  & \cdots & -2\lambda_{i-1}+1& 1& t_{i-1,i+1}^n+1&\cdots& t_{i-1,N-1}^n+1& t_{i-1,N}^n\\
	0 & 0  & \cdots & 0& 1& t_{i,i+1}^n+1&\cdots& t_{i,N-1}^n+1& t_{i,N}^n\\
	0 & 0  & \cdots & 0& 1& -2\lambda_{i+1}+1&\cdots& t_{i+1,N-1}^n+1& t_{i+1,N}^n\\
	\vdots & \vdots  & \ddots & \vdots& \vdots& \vdots&\ddots& \vdots& \vdots\\
	0 & 0  & \cdots& 0 & 1& 0& \cdots& -2\lambda_{N-1}+1&t_{N-1,N}^n\\
	0 & 0 &\cdots& 0 & 1& 0& \cdots& 0&-2\lambda_N
	\end{vmatrix}\\
	&=\prod_{j=1}^{i-1}\(-2\lambda_j+1\)\begin{vmatrix}
	0& t_{i,i+1}^n+1&\cdots& t_{i,N-1}^n+1& t_{i,N}^n+2\lambda_N\\
	0& -2\lambda_{i+1}+1&\cdots& t_{i+1,N-1}^n+1& t_{i+1,N}^n+2\lambda_N\\
	\vdots& \vdots&\ddots& \vdots& \vdots\\
	0& 0& \cdots& -2\lambda_{N-1}+1&t_{N-1,N}^n+2\lambda_N\\
	1& 0& \cdots& 0&-2\lambda_N
	\end{vmatrix}\\
	&=(-1)^{N-i}\prod_{j=1}^{i-1}\(-2\lambda_j+1\)K_N^{i,i+1}(n).
	\end{aligned}
	\]
	Note that $\(\bm{e}^T \Upsilon_{N}^{(n)} (\mathbb M_{N}^{(n)})^{*}\)_i$  is equal to the determinant after replacing the $i$-th row of the matrix $\mathbb M_{N}^{(n)}$  with the vector $\bm{e}^T \Upsilon_{N}^{(n)}$.
	In a similar manner, one can also derive that
	\[
	\(\bm{e}^T \Upsilon_{N}^{(n)} (\mathbb M_{N}^{(n)})^{*}\)_i
	=\prod_{j=1}^{i-1}\(-2\lambda_j-1\)L_N^{i,i+1}(n).
	\]
	The proof is complete.
\end{proof}
\begin{thm}\label{thm3.31}
	Let $u_j$ be the solution to \eqref{eq:mainmd02}, with $N$-layer concentric disks $A^{(j)}$, $j=1,2,$ respectively. Let $\Omega$ be a bounded domain enclosing $A^{(1)}\cup A^{(2)}$ and $H = \sum_{n=1}^\infty a_nr^ne^{\mathrm{i}n\theta}$, where $a_n\neq 0 $ for every sufficiently large $n$. If $u_1=u_2$ on $\Pi$, then
	\[
	r^{(1)}_k=r^{(2)}_k \mbox{ and } \sigma^{(1)}_k=\sigma^{(2)}_k,\; k=1,2,\ldots,N,
	\]
	where $\Pi$ is an open subset of $\p \Omega$.
\end{thm}
\begin{proof}[\bf Proof]
Since $u_1=u_2$ on $\Pi$, by using unique continuation, it is easy to see that $u_1=u_2$ in $\RR^2\setminus (A^{(1)}\cup A^{(2)})$. Then by applying Theorem \ref{th:02}, the coincidence of the locations of $N$-layer concentric disks can be obtained.
Without loss of generality, assume that $r_1^{(1)}>r_1^{(2)}$.
It follows from \eqref{eq:purbmn01} that for $r\geqslant r_1^{(1)}$,
\begin{equation}\label{eq:BM}
\begin{aligned}
& \left(\frac{(r_1^{(1)})^{2n}}{r^n},\frac{(r_2^{(1)})^{2n}}{r^n},\frac{(r_3^{(1)})^{2n}}{r^n},\ldots,\frac{(r_N^{(1)})^{2n}}{r^n}\right) (\mathbb M_{N,1}^{(n)})^{-1}\bm{e}\\
=&\left(\frac{(r_1^{(2)})^{2n}}{r^n},\frac{(r_2^{(2)})^{2n}}{r^n},\frac{(r_3^{(2)})^{2n}}{r^n},\ldots,\frac{(r_N^{(2)})^{2n}}{r^n}\right) (\mathbb M_{N,2}^{(n)})^{-1}\bm{e}.
\end{aligned}
\end{equation}
Taking $r = r_1^{(1)}$ and dividing $(r_1^{(1)})^{n}$ on the both sides of the above equality, one has that
\begin{equation}\label{DR112n}
\begin{aligned}
& \left(1,\frac{(r_2^{(1)})^{2n}}{(r_1^{(1)})^{2n}},\frac{(r_3^{(1)})^{2n}}{(r_1^{(1)})^{2n}},\ldots,\frac{(r_N^{(1)})^{2n}}{(r_1^{(1)})^{2n}}\right) (\mathbb M_{N,1}^{(n)})^{-1}\bm{e}\\
=&\left(\frac{(r_1^{(2)})^{2n}}{(r_1^{(1)})^{2n}},\frac{(r_2^{(2)})^{2n}}{(r_1^{(1)})^{2n}},\frac{(r_3^{(2)})^{2n}}{(r_1^{(1)})^{2n}},\ldots,\frac{(r_N^{(2)})^{2n}}{(r_1^{(1)})^{2n}}\right) (\mathbb M_{N,2}^{(n)})^{-1}\bm{e}.
\end{aligned}
\end{equation}
Note that
\[
\lim_{n\to\infty} \mathbb{M}_N^{(n)} =  \mathbb{M}_N := \begin{bmatrix}
-2\Gl_1 & 0 & 0 & \cdots& 0 &0 \\
-1 & -2\Gl_{2} & 0 &\cdots& 0  & 0 \\
-1 & -1 & -2\Gl_{3} &\cdots& 0  & 0 \\
\vdots & \vdots & \vdots & \ddots & \vdots& \vdots\\
-1 & -1 & -1 & \cdots & -2\Gl_{N-1}&0 \\
-1 & -1& -1 &\cdots & -1&-2\Gl_{N}
\end{bmatrix}.
\]
It follows from \eqref{DR112n} that, for $n$ large enough
\[
\left(1,0,0,\ldots,0\right) (\mathbb M_{N,1})^{-1}\bm{e}=0,
\]
which implies that $\left((\mathbb M_{N,1})^{*}\bm{e}\right)_1=0$. On the other hand, from \eqref{pei}, we have
\[
 \left((\mathbb M_{N,1})^{*}\bm{e}\right)_1 =(-2)^{N-1} \prod_{i=2}^{N}\lambda_{i}^{(1)} =\prod_{i=2}^{N} \frac{\sigma_{i}^{(1)}+\sigma_{i-1}^{(1)}}{(\sigma_{i-1}^{(1)}-\sigma_{i}^{(1)})}\neq 0,
\]
which is a contradiction. Hence
\[
r_1:=r_1^{(1)}=r_1^{(2)}.
\]
Using  \eqref{DR112n} again, we can obtain that
	\begin{equation}\label{sigma1}
	\frac{1}{-2\Gl_1^{(1)}}=\left((\mathbb M_{N,1})^{-1}\bm{e}\right)_1 = \lim_{n\to\infty}\left((\mathbb M^{(n)}_{N,1})^{-1}\bm{e}\right)_1 = \lim_{n\to\infty}\left((\mathbb M^{(n)}_{N,2})^{-1}\bm{e}\right)_1 =\left((\mathbb M_{N,2})^{-1}\bm{e}\right)_1 = \frac{1}{-2\Gl_1^{(2)}},
	\end{equation}
thus, $\Gl_1:=\Gl_1^{(1)} = \Gl_1^{(2)}$, i.e.,
\[
\sigma^{(1)}_1=\sigma^{(2)}_1.
\]
We next assume that $r_2^{(1)}>r_2^{(2)}$.
It follows from unique continuation and \eqref{eq:purbmn01} that for $r_2^{(1)} \leqslant r< r_1^{(1)}$, 
\begin{equation}\label{eq:BM2}
\begin{aligned}
& \left({r^n},\frac{(r_2^{(1)})^{2n}}{r^n},\frac{(r_3^{(1)})^{2n}}{r^n},\ldots,\frac{(r_N^{(1)})^{2n}}{r^n}\right) (\mathbb M_{N,1}^{(n)})^{-1}\bm{e}\\
=&\left({r^n},\frac{(r_2^{(2)})^{2n}}{r^n},\frac{(r_3^{(2)})^{2n}}{r^n},\ldots,\frac{(r_N^{(2)})^{2n}}{r^n}\right) (\mathbb M_{N,2}^{(n)})^{-1}\bm{e}.
\end{aligned}
\end{equation}
Taking $r = r_2^{(1)}$ and dividing $(r_2^{(1)})^{n}$ on the both sides of \eqref{eq:BM2}, one has that
\begin{equation}\label{eq:BM2r21}
\begin{aligned}
& \left(1,1,\frac{(r_3^{(1)})^{2n}}{(r_2^{(1)})^{2n}},\ldots,\frac{(r_N^{(1)})^{2n}}{(r_2^{(1)})^{2n}}\right) (\mathbb M_{N,1}^{(n)})^{-1}\bm{e}\\
=&\left(1,\frac{(r_2^{(2)})^{2n}}{(r_2^{(1)})^{2n}},\frac{(r_3^{(2)})^{2n}}{(r_2^{(1)})^{2n}},\ldots,\frac{(r_N^{(2)})^{2n}}{(r_2^{(1)})^{2n}}\right) (\mathbb M_{N,2}^{(n)})^{-1}\bm{e}.
\end{aligned}
\end{equation}
It follows from \eqref{eq:BM2r21} that for  $n$ large enough
\[
\left((\mathbb M_{N,1})^{-1}\bm{e}\right)_1+\left((\mathbb M_{N,1})^{-1}\bm{e}\right)_2=\left((\mathbb M_{N,2})^{-1}\bm{e}\right)_1,
\]
this, together with \eqref{sigma1}, implies that $\left((\mathbb M_{N,1})^{-1}\bm{e}\right)_2 =0$. However,
\[
\left((\mathbb M_{N,1})^{*}\bm{e}\right)_2 = (-2)^{N-2} (-2\lambda_{1}+1) \prod_{i=3}^{N}\lambda_{i}^{(1)}\neq 0.
\]
This is a contradiction. Hence
\[
r_2:=r_2^{(1)}=r_2^{(2)}.
\]
Using  \eqref{eq:BM2r21} again, we can obtain that
\[
\frac{-2\lambda_{1}+1}{4\Gl_1 \Gl_2^{(1)}}=\left((\mathbb M_{N,1})^{-1}\bm{e}\right)_2  =\left((\mathbb M_{N,2})^{-1}\bm{e}\right)_2 = \frac{-2\lambda_{1}+1}{4\Gl_1 \Gl_2^{(2)}},
\]
thus, $\Gl_2:=\Gl_2^{(1)} = \Gl_2^{(2)}$, i.e.,
\[
\sigma^{(1)}_2=\sigma^{(2)}_2.
\]
Analogously, since
\[
\left((\mathbb M_{N})^{*}\bm{e}\right)_k = (-2)^{N-k}\prod_{i=1}^{k-1} (-2\lambda_{i}+1) \prod_{i=k+1}^{N}\lambda_{i}\neq 0,
\]
we can conclude that
\[
r_k^{(1)}=r_k^{(2)} \mbox{ and } \sigma^{(1)}_k=\sigma^{(2)}_k,\quad k=3,4,\ldots,N.
\]
The proof is complete.
\end{proof}
\begin{rem}
 It is known that the $N$-layer concentric disks can be achieved as GPT-vanishing structure of $N-1$ order (see \cite{AKLL11}).
Theorem \ref{thm3.31} shows that  the conductivity distribution of the multi-layer concentric disks can be uniquely recovered under high-order probing wave.
Indeed, this is also physically justifiable.
\end{rem}
\begin{rem}
We remark that Theorem \ref{thm3.31} may be generalized to prove the uniqueness of layers $N$, although we fixed it in the theorem. Indeed, if we suppose $A^{(1)}=\cup_{k=1}^{N_1}A_k^{(1)}$ and $A^{(2)}=\cup_{k=1}^{N_2}A_k^{(2)}$, and without loss of generality, assume that $N_1>N_2$.
It follows from \eqref{eq:purbmn01} that
\[
\begin{aligned}
& \left(\frac{(r_1^{(1)})^{2n}}{r^n},\frac{(r_2^{(1)})^{2n}}{r^n},\ldots,\frac{(r_{N_2}^{(1)})^{2n}}{r^n},\frac{(r_{N_2+1}^{(1)})^{2n}}{r^n},\ldots,\frac{(r_{N_1}^{(1)})^{2n}}{r^n}\right) (\mathbb M_{N_1,1}^{(n)})^{-1}\bm{e}_{N_1\times 1}\\
=&\left(\frac{(r_1^{(2)})^{2n}}{r^n},\frac{(r_2^{(2)})^{2n}}{r^n},\ldots,\frac{(r_{N_2}^{(2)})^{2n}}{r^n}\right) (\mathbb M_{N_2,2}^{(n)})^{-1}\bm{e}_{N_2\times 1}.
\end{aligned}
\]
which implies that
\[
\begin{aligned}
& \left(\frac{(r_1^{(1)})^{2n}}{r^n},\frac{(r_2^{(1)})^{2n}}{r^n},\ldots,\frac{(r_{N_2}^{(1)})^{2n}}{r^n},\frac{(r_{N_2+1}^{(1)})^{2n}}{r^n},\ldots,\frac{(r_{N_1}^{(1)})^{2n}}{r^n}\right) (\mathbb M_{N_1,1}^{(n)})^{-1}\bm{e}_{N_1\times 1}\\
=&\left(\frac{(r_1^{(2)})^{2n}}{r^n},\frac{(r_2^{(2)})^{2n}}{r^n},\ldots,\frac{(r_{N_2}^{(2)})^{2n}}{r^n},0,\ldots,0\right)_{1\times N_1} (\widetilde{\mathbb{M}}_{N_1,2}^{(n)})^{-1}((\bm{e}_{1\times N_2 }, 0,\ldots,0)_{1\times N_1})^T,
\end{aligned}
\]
where the matrix $\widetilde{\mathbb{M}}_{N_1,2}^{(n)} $ is formed by replacing the first $N_2$ rows and $N_2$ columns in  identity matrix  ${\mathbb{I}}_{N_1}$ with ${\mathbb{M}}_{N_2,2}^{(n)}$.
Using the similar strategy as in Theorem \ref{thm3.31}, one can show that
\[
r_k^{(1)}=r_k^{(2)} \mbox{ and } \sigma^{(1)}_k=\sigma^{(2)}_k,\quad k=1,2,\ldots,N_2\;\mbox{ and }r_k^{(1)}=0,\quad k=N_2+1,\ldots,N_1,
\]
which means that $N_1 = N_2$.
\end{rem}

\subsection{Uniqueness of the conductivity value for multi-layer concentric disks}
By Theorem \ref{thm3.31}, we see that the conductivity distribution of the multi-layer concentric disks can be uniquely recovered  under high-order probing wave.  Next, we shall give the unique recovery of the material information, (i.e., $\sigma_{k}, k=1,2,\ldots,N$) with the known geometric information under some low-order probing waves. 
To this end, we first introduce the following notation.  Let $M> N$.
We denote by $C^{N}_{M}$ the set of all combinations of $N$ out $M$,  say e.g., for one combination
\[
(i_1, i_2, \ldots, i_{N})\in C^{N}_{M}\quad \mbox{satisfying} \quad
1\leqslant i_1<i_2<\cdots< i_{N}\leqslant M.
\]
In what follows, we also need the following assumption.
\begin{asm}\label{asm51}
Assume that there exists the partial-order background electrical potential
\begin{equation}\label{equ530}
	H = \sum_{k=1}^{N} a_{i_k}r^{i_k} e^{\mathrm{i} i_k\theta},
\end{equation}
	where $(i_1, i_2, \ldots, i_{N})\in C^{N}_{M}$, such that the matrix $\mathbb{T}_N$ is invertible, where $\mathbb{T}_N := \((\mathbb{L}_{N,1}))_{k,l}(\mathbb{R}_{N,2})_{l,k}\)_{k,l=1}^N$,
\begin{equation}\label{MNN}
\mathbb{L}_{N,1}:=\begin{bmatrix}
\bm{e}^T \Upsilon_{N}^{(i_1)} (\mathbb M_{N,1}^{(i_1)})^{*}\\ \bm{e}^T \Upsilon_{N}^{(i_2)} (\mathbb M_{N,1}^{(i_2)})^{*} \\\vdots\\\bm{e}^T \Upsilon_{N}^{(i_N)} (\mathbb M_{N,1}^{(i_N)})^{*}
\end{bmatrix},\;\mbox{ and  }\;
\mathbb{R}_{N,2}:=\begin{bmatrix}
(\mathbb M_{N,2}^{(i_1)})^{*}\bm{e} & (\mathbb M_{N,2}^{(i_2)})^{*}\bm{e}  &\cdots& (\mathbb M_{N,2}^{(i_N)})^{*}\bm{e}
\end{bmatrix}.
\end{equation}
\end{asm}
\begin{thm}\label{thm3.3}
Let $u_j$ be the solution to \eqref{eq:mainmd02}, with conductivity  ${\sigma}_k^{(j)}$, $j=1,2,$ respectively. Let $\Omega$ be a bounded domain enclosing $A=\cup_{k=1}^{N}A_k$, i.e., $A\subset \Omega$, and let  $H$ be given by \eqref{equ530}.  If  Assumption \ref{asm51} and  $u_1=u_2$ on $\Pi$ hold, then
\[
 \sigma^{(1)}_k=\sigma^{(2)}_k,\; k=1,2,\ldots,N,
\]
where $\Pi$ is an open subset of $\p \Omega$.
\end{thm}

\begin{proof}[\bf Proof]
	Since $u_1=u_2$ on $\Pi$, by using unique continuation, it is easy to see that $u_1=u_2$ in ${A}_0$.
	It follows from \eqref{eq:purbmn01} and Assumption \ref{asm51}  that
	\[
	\bm{e}^T\Upsilon_{N}^{(n)} (\mathbb M_{N,1}^{(n)})^{-1}\bm{e} = \bm{e}^T\Upsilon_{N}^{(n)} (\mathbb M_{N,2}^{(n)})^{-1}\bm{e}, \mbox{ for } n = i_1,i_2,\ldots,i_N,
	\]
	which implies that
	\begin{equation}\label{u1=u2}
	\bm{e}^T \Upsilon_{N}^{(n)} (\mathbb M_{N,1}^{(n)})^{*}  \(\mathbb M_{N,2}^{(n)}-\mathbb M_{N,1}^{(n)}\)(\mathbb M_{N,2}^{(n)})^{*}\bm{e}=0, \mbox{ for } n = i_1,i_2,\ldots,i_N.
	\end{equation}
	Then by applying Theorem \ref{thm3.31}, we can obtain
	\[
	\mathbb M_{N,2}^{(n)}-\mathbb M_{N,1}^{(n)}=
	\begin{bmatrix}
	2\lambda^{(1)}_{1}-2\lambda^{(2)}_{1} & 0 & \cdots & 0 \\
	0 &2\lambda^{(1)}_{2}-2\lambda^{(2)}_{2}  & \cdots & 0\\
	\vdots & \vdots &\ddots &\vdots \\
	0 & 0 & \cdots & 2\lambda^{(1)}_{N}-2\lambda^{(2)}_{N}
	\end{bmatrix},
	\]
	which implies that $\mathbb M_{N,2}^{(n)}-\mathbb M_{N,1}^{(n)}$  is  independent of  the choice of $n$.
	Note that \eqref{u1=u2}  can be rewritten as
	\[
	\mathbb{T}_N \(2\lambda^{(1)}_{1}-2\lambda^{(2)}_{1},2\lambda^{(1)}_{2}-2\lambda^{(2)}_{2},\ldots,2\lambda^{(1)}_{N}-2\lambda^{(2)}_{N}\)^T = 0.
	\]
	Since the matrix $\mathbb{T}_N$ is invertible, we get
	\[
	\lambda^{(1)}_{k}=\lambda^{(2)}_{k},\; k=1,2,\ldots,N.
	\]
    In view of \eqref{lamdk}, we have
	\[
	\sigma^{(1)}_k=\sigma^{(2)}_k,\; k=1,2,\ldots,N.
	\]
	The proof is complete.
\end{proof}

\begin{rem}
Next we want to show that the restriction on the invertibility of the matrix $\mathbb{T}_N$ is not difficult to achieve, we shall present an example in what follows.
In view of \eqref{lamdk}, we have that
\[
\lambda^{(j)}_k\in (-\infty,-1/2)\cup(1/2,+\infty),\;k = 1,2,\ldots, N, \;j = 1,2.
\]
For two-layer structure, by taking $(i_1,i_2)=(1,2)$ in \eqref{MNN},  and by using Lemmas \ref{recukl}--\ref{gtf}, we have that
\[
\begin{aligned}
\mathbb{L}_{2,1}=\begin{bmatrix}
-2\lambda^{(1)} _{2} r_{1}^2+r_{2}^2 & -2\lambda^{(1)} _{1} r_{2}^2-r_{2}^2\\ -2\lambda^{(1)} _{2} r_{1}^4+r_{2}^4 & -2\lambda^{(1)} _{1} r_{2}^4-r_{2}^4 \end{bmatrix},
\end{aligned}
\]
and
\[
\begin{aligned}
\mathbb{R}_{2,2}= \begin{bmatrix} -2\lambda^{(2)} _{2}-t_{1,2} & -2\lambda^{(2)} _{2}-t^2_{1,2}\\ 1-2\lambda^{(2)} _{1} & 1-2\lambda^{(2)} _{1} \end{bmatrix}.
\end{aligned}
\]
By the definition of $\mathbb{T}_N$ and direct computations, we have that
\[
|\mathbb{T}_2| = -\frac{r_{2}^2 \left(2\lambda^{(1)} _{1}+1\right) \left(2\lambda^{(2)} _{1}-1\right) \left(4\lambda^{(1)} _{2}\lambda^{(2)} _{2}r_{1}^6+r_{2}^6\right) \left(r_{1}^2-r_{2}^2\right)}{r_{1}^4}\neq  0,
\]
which implies that the matrix $\mathbb{T}_2$ is invertible.
\end{rem}

\section{Concluding remarks}\label{s5}
In this paper, we derived the asymptotic expansions for the electric potential field in presence of a multi-layer structure. We also showed some properties of the induced GPTs. When the multi-layer structure satisfies the symmetry property, we derived the exact formulation of the GPTs, which is reduced to the so-called Generalized Polarization Matrix. With the help of such formulation, we were able to show the unique recovery results for both the structures and the conductivities by using only one \emph{partial order} measurement. Stability and numerical implementations for reconstructing such multi-layer structure will be our forth coming works.

\section*{Acknowledgement}
The authors wish to thank the anonymous referees for the constructive and insightful comments and suggestions, which have led to significant improvements on the presentation and results of this paper.
The work of Y. Deng was supported by NSFC-RGC Joint Research Grant No. 12161160314.

\section*{Data availability statement}
All data generated or analysed during this study are included in this published article.

\end{document}